\documentclass[a4paper,twosided, 11pt]{article}
\usepackage[utf8]{inputenc} 
\usepackage[T1]{fontenc} 
\usepackage[a4paper]{geometry} 
\usepackage{amsmath, amssymb, mathtools, amsthm}
\usepackage{mathtools} 
\usepackage{mathrsfs}  
\usepackage{graphicx} 
\usepackage{xcolor}   
\renewcommand{\color}[2][]{}
\usepackage[english]{babel} 
\usepackage{bbm}
\usepackage[title]{appendix}
\newcommand{\eps}{\varepsilon}
\usepackage{caption}
\usepackage{subcaption}
\usepackage{comment}
\usepackage{geometry}
\usepackage{cite}
\usepackage{cancel}
\usepackage{float}
\usepackage{booktabs}
\usepackage{varwidth}
\usepackage{enumitem} 
\usepackage[colorlinks=true]{hyperref} 

\newcommand \commentout[1] {}

\newcommand{\R}{\mathbb{R}}

\newcommand {\Chi} {{\bf \raise 2pt \hbox{$\chi$}} }

\newcommand {\Div}  { {\rm div} }

\newcommand {\f}   {\frac}
\newcommand {\p}   {\partial}

\newcommand*{\dd}{\mathop{\kern0pt\mathrm{d}}\!{}}

\newcommand*{\DD}{\mathop{\kern0pt\mathrm{D}}\!{}}

\DeclareMathOperator*{\argmin}{argmin}

\theoremstyle{plain}
\newtheorem*{thm*}{Theorem}
\newtheorem{thm}{Theorem}[section]

\newtheorem{lemma}[thm]{Lemma}
\newtheorem{proposition}[thm]{Proposition}
\newtheorem{corollary}[thm]{Corollary}

\theoremstyle{remark}
\newtheorem{remark}[thm]{\bf Remark}
\newtheorem{definition}[thm]{\bf Definition}

\newcommand{\beq}{\begin{equation}}
\newcommand{\eeq}{\end{equation}}
\newcommand{\bea} {\begin{array}{rl}}
\newcommand{\eea} {\end{array}}
\newcommand{\bepa}{\left\{ \begin{array}{l}}
\newcommand{\eepa} {\end{array}\right.}
\newcommand{\diff}{\mathop{}\!\mathrm{d}}

\numberwithin{equation}{section}

\title{Sobolev estimates for the Keller-Segel system and applications to the JKO scheme}

\author{Charles Elbar\thanks{Université Claude Bernard Lyon 1, ICJ UMR5208, CNRS, Ecole Centrale de Lyon, INSA Lyon, Université Jean Monnet, 69622
Villeurbanne, France. Email: elbar@math.univ-lyon1.fr}
}

\date{}

\begin{document}

\maketitle

\begin{abstract}
We prove $L^{\infty}_{t}W^{1,p}_{x}$ Sobolev estimates in the Keller-Segel system {\color{blue} with linear diffusion in any dimension} by proving a functional inequality, inspired by the Brezis-Gallouët-Wainger inequality. These estimates are also valid at the discrete level in the Jordan-Kinderlehrer-Otto (JKO) scheme. By coupling this result with the diffusion properties of a functional according to Bakry-Emery theory, we deduce the $L^2_t H^{2}_{x}$ convergence of the scheme, thereby extending the recent result of Santambrogio and Toshpulatov in the context of the Fokker-Planck equation to the Keller-Segel system.

\end{abstract}
\vskip .7cm

\noindent{\makebox[1in]\hrulefill}\newline

\section{Introduction}

We are concerned with estimates and convergence of the Jordan-Kinderlehrer-Otto (JKO) scheme for the Keller-Segel (KS) system
\begin{equation}\label{eq:KS}
\partial_t \rho - \Delta \rho + \Div(\rho\nabla u[\rho]) = 0, \quad -\Delta u[\rho] =\rho.
\end{equation}

The system is set on the domain $(0,T)\times\Omega$, where $T > 0$ and $\Omega\subset \R^d$ is a smooth, convex, bounded open set. The boundary of $\Omega$ is denoted by $\partial\Omega$, with outward unit normal vector $\vec{n}$. It is complemented by a nonnegative initial condition $\rho(0,\cdot) = \rho_0$ which is assumed to be smooth enough. The equation is supplemented with the no flux and Dirichlet boundary conditions:

$$
{\color{blue}(\nabla\rho-\rho \nabla u[\rho])\cdot\vec{n} = u[\rho] = 0} \quad \text{on } \partial \Omega.
$$

Our objective is to prove that {\color{blue}in a subcritical case defined by the initial data, the solutions of the Keller-Segel exist globally}, and that they are bounded in $L^{\infty}(0,T; W^{1,p}(\Omega))$ for all $1 \leq p < \infty$ by a constant depending only on $T$ and the initial condition $\rho_{0}$, and that these estimates are also available in the JKO scheme. Finally we prove that the curve defined by the scheme converges strongly in the space $L^2(0,T; H^{2}(\Omega))$. \\

\textbf{Chemotaxis and the Keller-Segel system.}  Chemotaxis refers to the ability of cells to detect chemical signals in their environment and change their movement accordingly. It was first identified in bacterias by Engelmann and Pfeffer in the 1880s. This process plays a critical role in many biological applications, including immune responses, wound healing, and microbial behavior.

Positive chemotaxis refers to movement toward a higher concentration of an attractive substance, such as nutrients, while negative chemotaxis describes movement away from harmful substances, like toxins. For example, \textit{Escherichia coli} and more generally bacterias use positive chemotaxis to locate and move toward nutrients~\cite{adler} and negative chemotaxis to move away from repellents~\cite{wai-adler}. 

The theoretical and mathematical study of chemotaxis began with the pioneering work of Patlak in the 1950s~\cite{patlak1953random} and was further developed by Keller and Segel in the 1970s~\cite{keller_initiation_1970,KELLER1971225}. In their work, they derived a system known as the Keller-Segel system which combines diffusion and drift terms that account for nonlocal interactions. The general form of the system is given by:
$$
\p_t \rho - \Div(\phi(\rho,u)\nabla \rho) +\Div( \psi(\rho,u)\nabla u) = f(\rho,u),\quad \tau\p_t u = \Delta u + g(\rho,u) - h(\rho,u), 
$$
where $\rho$ denotes the local density of cells, and $u$ represents the concentration of chemical signals. The term $\Div(\phi(\rho,u)\nabla \rho)$ models diffusion, which may be nonlinear, while $\Div(\psi(\rho,u)\nabla u)$ describes the drift due to nonlocal interactions, such as attraction and repulsion between cells. 

In this paper, we focus on a simplified, classical version of the Keller-Segel (KS) system, where we assume $\phi(\rho,u) = 1$ (linear diffusion), $\psi(\rho,u) = \rho$, $f(\rho,u) = h(\rho,u) = 0$, $\tau = 0$, and $g(\rho,u) = \rho$. This results in the well-known elliptic form of the KS system. For the reader interested in the analytic properties of more general cases of the Keller-Segel system, we refer to the reviews~\cite{Arumugam2020KellerSegelCM, doi:10.1142/S021820251550044X,hillen2009user,horstmann2004keller}.

An important phenomenon in the Keller-Segel system is the possibility for solutions to blow-up in finite time. This happens due to the competition between aggregation, that is the attraction of cell and diffusion. If the aggregative forces are sufficiently strong, they can overcome diffusion, leading to blow-up. In particular, it was shown by Blanchet, Dolbeault and Perthame~\cite{blanchet-keller-2006} in two dimension that if the initial mass of the system is greater than $8\pi$, the solution will blow up in finite time. More precisely, Nagai~\cite{NAGAI2001777} proved that at a blow up point $x_{0}$, there exists a nonnegative integrable function $f$, $\eps>0$ and a constant $m$ such that 
$$
\lim_{t\to T_{max}}\rho(\cdot , t) = m\delta_{x_{0}} + f \quad \text{weak star in $M(\Omega(x_{0},\eps))$}
$$
where $\Omega(x_{0},\eps)= \{x\in \bar{\Omega}, |x-x_0|\le \eps\}$ and $T_{max}$ is the final time of blow up. This reasoning can be extended to other dimensions, see~\cite{NAGAI19973837,perthame2006transport,Carrillo2018AggregationDiffusionED}.  On the other hand, when the initial mass is below this threshold, solutions are known to exist globally. 

In this paper, we {\color{blue} consider a small-data subcritical regime, quantified by the $L^{\f d2}$ norm, preventing blow-up as in~\cite{MR2099126}}. Our first goal is to derive Sobolev estimates in this regime. Our second goal is to construct solutions using the JKO scheme and prove convergence of this scheme in high Sobolev norms.

To construct weak solutions for the Keller-Segel system, a common approach is to begin with a regularized version of the system, where, for instance, the singularity of the kernel is truncated appropriately with a parameter $\eps$ (we recall that on the whole space $u=-W\ast \rho$ where $W(x) = -\f{1}{2\pi} \log |x|$ in two dimensions). Existence of solutions for this regularized system, which can be seen as an aggregation diffusion system with a smooth kernel, can be obtained using fixed-point methods. It remains to remove the regularization parameter, that is to pass to the limit $\eps \to 0$ using compactness methods. An alternative strategy to prove existence is to observe that the system can be viewed as a gradient flow with respect to the Wasserstein metric~\cite{blanchet2013gradient}, and use the JKO scheme to construct solutions.
 \\

\textbf{The JKO scheme.} In 1998, Jordan, Kinderlehrer and Otto~\cite{MR1617171} constructed solutions to the Fokker-Planck equation
$$
\p_t \rho - \Delta\rho + \Div(\rho\nabla V) = 0, \quad \text{$V(x)$ given and smooth,}
$$
with a new scheme. They observed that the equation is the gradient flow of the functional
$$
F(\rho) = \int_{\Omega}  \rho\log \rho - \rho V \diff x,
$$

in the Wasserstein space. Indeed, the equation can be rewritten 
$$
\p_t \rho - \Div\left(\rho \nabla \f{\delta F}{\delta \rho}[\rho]\right)=0, 
$$

where $\f{\delta F}{\delta \rho}[\rho]=\log \rho + V$ is the first variation of $F$.  This interpretation also naturally leads to a time discretization.  For a fixed time step \(\tau > 0\), one can define a sequence $\{\rho_\tau^n\}_{n}$ iteratively through $\rho_{\tau}^{0}= \rho_{0}$ and

\[
\rho_\tau^{n+1} \in \arg\min_{\rho} \left\{ F(\rho) + \frac{W_2^2(\rho, \rho_\tau^n)}{2\tau} \right\}.
\]

This formulation is reminiscent of the implicit Euler scheme for Euclidean gradient flows. This sequence can then be used to define the constant interpolation curve \( t \mapsto \rho_\tau(t) \) in the space of probability measures, with \(\rho_\tau(0) = \rho_0\) and 

\[
\rho_\tau(t) = \rho_\tau^{n+1} \text{ for } t \in (n\tau, (n+1)\tau].
\]
In their work, they showed that the constructed curve $\rho_{\tau}$ converges strongly in $L^{1}_{t,x}$ and weakly in $L^{1}_{x}$ for all time, to a weak solution of the Fokker-Planck equation. Since then, there has been extensions of this scheme to many equations, including  Keller-Segel~\cite{carrillo2018,blanchet2013gradient,blanchet2008convergence, dimarino,kim2023density}, nonlocal equations~\cite{burger2023porous,carrillo2024nonlocal,di2018nonlinear}, Cahn-Hilliard~\cite{antonio2024competing,lisini2012cahn,kroemer2022hele,carrillo2023degenerate} etc.  

The convergence result of the JKO scheme for the Fokker-Planck equation was then improved by Santambrogio and Toshpulatov in~\cite{toshpulatov}. Although the rate of convergence remains the same, their results proves convergence in stronger norms, specifically in $L^{2}_{t}H^{2}_{x}$. The core of their analysis relies on the five-gradients inequality, initially introduced in~\cite{de2016bv} to obtain $BV$ estimates for some variational problems arising in optimal transport. This inequality has been generalized to other costs~\cite{CRMATH_2023__361_G3_715_0} and on differentiable manifolds~\cite{DIMARINO2024294}. Below, we briefly outline the key ideas of the argument from~\cite{toshpulatov}, as we follow a similar strategy in our context.

At the continuous level, let 
\[
\mathcal{F}_{2}[\rho] = \int_{\Omega} \rho|\nabla (\log \rho + V)|^2 \diff x
\]
the Fisher information functional (appearing in the Bakry-Emery theory).
It can be observed that if $\rho$ solves the Fokker-Planck equation the following relation holds:
\[
\frac{d}{dt} \mathcal{F}_{2}[\rho] + \int_{\Omega} \rho |D^{2}\log\rho+V|^2 \diff x= \text{l.o.t.},
\]
where l.o.t. denotes lower-order terms.

At the discrete level, using the five-gradients inequality in the JKO scheme, a similar estimate can be obtained but up to an error term. That is 
$$
\mathcal{F}_{2}[\rho_{\tau}(0)]\ge \mathcal{F}_{2}[\rho_{\tau}[T]] + \int_{0}^{T}\int_{\Omega} \rho_{\tau} |D^{2}\log\rho_{\tau}+V|^2 \diff x \diff t + \text{l.o.t.} + \eps(\tau).
$$

 Since it can be shown that $\rho_{\tau}$ converges strongly in $L^{\infty}_{t}H^1_{x}$ to $\rho$, it follows that:
\[
\limsup_{\tau\to 0}\int_{0}^{T}\int_{\Omega} \rho_{\tau} |D^{2}\log\rho_{\tau}+V|^2 \diff x \diff t  \le \int_{0}^{T}\int_{\Omega} \rho |D^{2}\log\rho+V|^2\diff x \diff t .
\]
Combining this estimate $L^{2}_{t}H^{2}_x$ estimates on $\rho_{\tau}$ and with the lower semi-continuity of the weak convergence, we deduce

\[
\int_{0}^{T}\int_{\Omega} \rho_{\tau} |D^{2}\log\rho_{\tau}+V|^2 \diff x \diff t \to \int_{0}^{T}\int_{\Omega} \rho |D^{2}\log\rho+V|^2 \diff x \diff t .
\]

Finally assuming a positive initial condition, and proving that a bound from below is preserved in time we conclude $\rho_{\tau} \to \rho$ strongly in $L^{2}_t H^{2}_x$. Indeed, the weak convergence and the convergence of the norms imply the strong convergence by uniform convexity.

The reasoning holds for $V$ a given function of $x$ assumed to be smooth enough. A natural question arises: Can we obtain the same result when replacing the given smooth potential $V$ with $W \ast \rho$, where $W$ is a general kernel? In this case, we are dealing with an aggregation-diffusion equation, which introduces additional technical challenges. For instance, the equation is now nonlinear and if $V = W \ast \rho$, the potential $V$ now depends on time. However, if $W$ remains smooth, the result should still hold.  In this paper, we consider a particular $W$, which is the one given by the Keller-Segel equation: $W$ is defined implicitly by the equation $-\Delta_{x} W(x,y) = \delta_{y}(x)$ and $W(x,y)=0$ for $x\in\p\Omega, \, y\in\Omega$. This kernel is not smooth as it has a singularity on the diagonal. Other non smooth kernel could be considered~\cite{2024arXiv240609227C}. A careful inspection of the proof in the Fokker-Planck case, shows that the same idea does not apply. Three difficulties need indeed to be overcome:

\begin{itemize}{\color{blue}
    \item \textbf{Difficulty 1: Existence and $L^{\infty}$ estimates in the JKO scheme in high dimensions.} In dimension $d\ge 3$, some difficulties arise already at the existence of the JKO scheme, because the negative part of the free energy cannot be controled by its positive part. Also, to obtain $L^{\infty}$ estimates on the solution, we need to adapt the Alikakos iteration method, but at the discrete level of the JKO scheme. }
    
    \item \textbf{Difficulty 2: Lower bound for $\rho_{\tau}$.} The second difficulty is to find a uniform lower bound for the curve $\rho_{\tau}$ away from zero. This estimate is crucial to deduce strong convergence in the $L^2_t H^2_x$ norm, with the dissipation of the functional $\mathcal{F}_2$, as detailed in the strategy above. For the Fokker-Planck equation, starting with a positive initial condition, such a lower bound can be obtained with a maximum principle, using the fact that $\|D^2 V\|_{L^{\infty}} \leq C$. Indeed, this maximum principle can also be applied at the discrete level of the JKO scheme.{\color{blue} In the Keller-Segel system we a priori don't know that $D^2 u$ is bounded in $L^{\infty}$.}
    
    \item \textbf{Difficulty 3: Uniform $C^{0,\alpha}(\bar{\Omega})$ estimates in the JKO scheme.} The third difficulty is to prove that the curves in the JKO scheme are uniformly bounded in $C^{0,\alpha}(\bar{\Omega})$ independently of the discretization parameter $\tau$. By interpolation, we can deduce that the first and second derivatives of the Kantorovich potential converge to zero as $\tau \to 0$ with a rate (see Lemma~\ref{lem:control_phi}). This convergence allows us to relate the dissipation of the functional $\mathcal{F}_2$ at the continuous level and at the level of the JKO scheme.
\end{itemize}

In the case of the Fokker-Planck equation, the third difficulty has been overcome in~\cite{toshpulatov} by using the $L^{\infty}_{t}W^{1,p}_{x}$ estimates for $p > d$ proved in~\cite{dimarino}. This implies the Hölder estimates by the embedding $W^{1,p}\hookrightarrow C^{0,\alpha}$. We also mention that in~\cite{caillet2024fisher}, the authors show that all the modulus of continuity can be controlled. The key ingredient to obtain the $W^{1,p}$ estimates is the functional
\begin{equation}\label{eq:FP}
    \mathcal{F}_{p}[\rho] = \int_{\Omega}\rho \left|\nabla \left(\log \rho - u[\rho]\right)\right|^p \,\diff x,
\end{equation}
which generalizes the functional $\mathcal{F}_{2}$ to other exponents. Here $u[\rho] = -V$ in the Fokker-Planck case and $u[\rho]$ satisfies the Poisson equation $-\Delta u[\rho] = \rho$ with Dirichlet boundary conditions in the Keller-Segel case.

More precisely, in~\cite{dimarino}, the authors proved that for the  Fokker-Planck equation, the $L^{\infty}_t W^{1,p}_{x}$ estimates hold for all $1 \leq p < +\infty$. However, for the Keller-Segel system in dimension 2, the estimates were only proved for $1 \leq p < 2$. Unfortunately, the $1\le p<2$ estimates are insufficient to deal with the two difficulties described above. Here we show that by using the same functional $\mathcal{F}_p$, as well as Theorem~\ref{thm:functional_inequality}, we can extend these estimates to cover all $1 \leq p < +\infty$. Then this additional regularity allows to prove the convergence of the scheme in the $L^2_t H^2_x$ norm by coming back to the functional $\mathcal{F}_2$ and applying the strategy described above from~\cite{toshpulatov}.

\subsection{Main results}

We present a functional inequality defined on a general smooth, bounded domain $\Omega$.

\begin{thm}[Functional Inequality]\label{thm:functional_inequality} Let $0<\alpha<1$. Then there exits a constant $C= C(d,\alpha, {\color{blue}\Omega})$  such that
\begin{equation}
    \|D^{2}u\|_{L^{\infty}} \leq C \|\Delta u \|_{L^{\infty}}\left(1+ \log\left(\frac{\|\Delta u\|_{C^{0,\alpha}}}{\|\Delta u \|_{L^{\infty}}}\right)\right).
\end{equation}
for all $u\in C^{2,\alpha}(\bar{\Omega})$ with $u=0$ on $\p\Omega$.
\end{thm}

\begin{remark}
The constant $C(d,\alpha)$ blows up when $\alpha\to 0$.
\end{remark}

\begin{corollary}\label{cor:functional_inequality}
 Let $0<\alpha<1$. Then there exits a constant $C= C(d,\alpha)$  such that
$$
 \|D^{2}u\|_{L^{\infty}} \leq C(1+ \|\Delta u \|_{L^{\infty}} (1+\log^{+}\left(\|\Delta u\|_{C^{0,\alpha}}\right))).
$$
\end{corollary}

Here $\log^{+}(x)=\log x$ if $x\ge 1$ and 0 else. This inequality shares some similarities with a result by Kozono and Taniuchi~\cite{kozono2000limiting}. To extend the blow-up criterion of solutions to the Euler equation, they have proved an $L^{\infty}$-estimate of functions in terms of the
BMO norm and the logarithm of a norm of higher derivatives. Their result, which is valid on the whole space, is
\begin{equation}
    \|f\|_{L^{\infty}(\R^d)} \leq C\left(1 + \|f\|_{\text{BMO}(\R^d)}\left(1 + \log^{+}\|f\|_{W^{s,p}(\R^d)}\right)\right), \quad sp > d.
\end{equation}

The Kozono-Taniuchi inequality itself draws inspiration from the Brezis-Gallouët-Wainger inequality~\cite{brezis1979nonlinear,brezis1980note}, which is expressed as:
\begin{equation}
    \|f\|_{L^{\infty}} \leq C\left(1 + \log^{\frac{r-1}{r}}\left(1 + \|f\|_{W^{s,p}}\right)\right), \quad sp > d,
\end{equation}
provided $\|f\|_{W^{k,r}}\le 1$, for $kr=d.$

Our result and these inequalities are an alternative way to bound the $L^{\infty}$ norm instead of using the classical embeddings $\|\cdot\|_{L^{\infty}}\le C\|\cdot\|_{W^{s,p}}$ or $\|\cdot\|_{L^{\infty}}\le C\|\cdot\|_{C^{0,\alpha}}$. Logarithmic control of the higher-order norms is much better in some situations. In particular, it is important in our context, as we expect to recover the regularity of the Keller-Segel system using Gronwall's lemma on the functional $\mathcal{F}_{p}$ defined in~\eqref{eq:FP}. In the resulting ODE, we find a term of the form 
$$
\f{d}{dt}\mathcal{F}_{p}(t)\lesssim f( \mathcal{F}_{p}(t))
$$
where $f$ should not grow too fast to $+\infty$, in order to have global estimates. This is the case of the function $x\mapsto x\log x$, where the logarithmic term appears from the functional inequality. 

The functional $\mathcal{F}_{p}$ controlling the $L^{\infty}_{t}W^{1,p}_{x}$ norms of the solutions, we are able to obtain the Sobolev estimates on the solutions of the Keller-Segel system. First let us define our notion of weak solutions to the Keller-Segel system.
{\color{blue}
\begin{definition}
$\rho$ is called a weak solutution of the Keller-Segel system~\eqref{eq:KS} on $[0,T)$ if:
\begin{itemize}
    \item $\rho\in L^{\infty}(0,T; L^{2}(\Omega))\cap L^{2}(0,T; H^{1}(\Omega))$,
    \item $u\in L^{2}(0,T; H^{1}_{0}(\Omega))$
    \item $(u,v)$ satisfies the equations for all $\varphi\in C^{\infty}((0,T)\times \Omega)$ and $\psi\in C_{c}^{\infty}((0,T)\times\Omega)$:
    \begin{align*}
    &\int_{0}^{T}\int_{\Omega}(\nabla \rho - \rho\nabla u) \cdot \nabla \varphi - \rho \partial_t\varphi \diff x \diff t= \int_{\Omega}\rho_{0}(x) \varphi(x,0)\diff x,\\
    &\int_{0}^{T}\int_{\Omega}\nabla u \cdot \nabla\psi\diff x \diff t =\int_{0}^{T}\int_{\Omega} \rho \psi \diff x \diff t
    \end{align*}
\end{itemize}
\end{definition}
}

\begin{thm}[Sobolev estimates]\label{thm:sobolev_bounds}
{\color{blue}Let $0<a \le \rho_0<b$ be a smooth initial condition such that $\mathcal{F}_{p}[\rho_0]<+\infty$ for some $p>d$. Then there exists $\eta_{d}$ depending on the dimension such that if $\|\rho_0\|_{L^{\frac{d}{2}}(\Omega)}\le \eta_d$, there exists a positive weak solution $\rho$ of the KS system~\eqref{eq:KS} with initial condition $\rho_0$. Moreover  $\rho$ is bounded in $L^{\infty}((0,T)\times\Omega)\cap L^{\infty}(0,T; W^{1,p}(\Omega))$.} More precisely there exists $C_1, C_2$ depending on $\|\rho(t)\|_{L^{\infty}}$ such that for all $t\in(0,T)$, 
$$
\|\rho(t)\|_{W^{1,p}}^{p} \le C_1(\mathcal{F}_{p}[\rho(t)] + 1),
$$ 
where $\mathcal{F}_{p}[\rho(t)]$ satisfy the dissipation inequality
$$
\f{d}{dt}\mathcal{F}_{p}[\rho(t)]\le C_2 \mathcal{F}_{p}[\rho(t)](\log(\mathcal{F}_{p}[\rho(t)]+1)+1). 
$$
In particular $\mathcal{F}_{p}[\rho(t)]\le (\mathcal{F}_{p}[\rho(0)]+1)^{\exp(C_2t)}\exp(\exp(C_2t)-1)-1.$
\end{thm}

{\color{blue}
\begin{remark}
Although this appears to be a natural result, we were not able to find in the literature a proof of the existence of a weak solution to~\eqref{eq:KS} that is bounded in $L^{\infty}((0,T)\times\Omega)$. It is even mentioned in~\cite{MR2099126} that this result is open. But such a bound should follow from an estimate on a high enough $L^{p}$ bound of $\rho$, which can easily be obtained provided the $L^{\frac{d}{2}}$-norm of the initial data is sufficiently small. Combining this estimate with the Alikakos iteration method should then yield the desired $L^{\infty}$ bound.

In fact, we use this strategy in the JKO scheme introduced below and we contruct a solution satisfying this regularity. For this reason, we do not repeat the argument here and simply assume the existence of a weak solution that is bounded in $L^{\infty}$. Note that the restriction on the initial data in the JKO scheme depends on the final time, but this can be removed for the continuous PDE in another approximating scheme as in~\cite{SUGIYAMA2006333}. The same argument holds for positivity: we prove it in the JKO scheme; but this would follow also on the equation by maximum principle. Our main focus is instead proving the $L^{\infty}(0,T; W^{1,p}(\Omega))$ bound, given the $L^{\infty}((0,T)\times\Omega)$ bound.
Even though the solution under consideration is only weak, all the computations presented here can be justified rigorously at the level of an appropriate approximation scheme, as in~\cite{SUGIYAMA2006333}. We therefore carry out the computations only at a formal level.

\end{remark}
}

    To state our result about the JKO scheme, we first introduce some notations. Let $\mathcal{F}: \mathcal{P}(\Omega)\to \R\cup\{+\infty\}$ be a functional defined on probability measures and $\rho_{0}$ an initial condition. Let $\tau>0$ be a fixed time step and $\eta\in \mathcal{P}(\Omega)$. Then, we denote by $Prox_{\mathcal{F}}^{\tau}(\eta)$ the solutions of the minimizing problem
$$
\argmin_{\rho\in\mathcal{P}(\Omega)} \left\{\mathcal{F}({\color{blue}\rho}) + \f{W^{2}_{2}(\rho,\eta)}{2\tau}\right\}.
$$
where the minimization runs across all $\rho$ with same mass than $\eta$.
The Keller-Segel equation is the gradient flow of the functional
\begin{equation}\label{eq:functional_gradientflow_KS}
\mathcal{J}(\rho) : = \int_{\Omega} \rho\log \rho -\f{1}{2} \rho u[\rho]\diff x.   
\end{equation}
{\color{blue}
However, it is not clear that a minimizer of this functional exists. The idea is to introduce a penalized scheme, where the artificial term is not really seen by the scheme.  We fix $M_d,\nu>0$ (to be chosen later) and define
\[
\widetilde{\mathcal{J}}[\rho]
 := 
\mathcal{J}[\rho] + \frac{1}{\nu}\Bigl(C_0\|\rho\|_{L^{\frac d2}(\Omega)}^2-M_d\Bigr)_+,
\qquad (x)_+:=\max\{x,0\},
\]
where $C_0$ is the best constant in the inequality 

\begin{equation}\label{eq:estimate_H1_Ld_rewritten}
\int_\Omega \rho\,u  = \int_\Omega |\nabla u|^2
 \le  C_0\,\|\rho\|_{L^{\frac d2}(\Omega)}^2.
\end{equation}

This constant exists by Lemma~\ref{lem:sobone2}.

We define the sequence of the JKO scheme, 
\begin{equation}\label{eq:def_JKO}
\rho_{\tau}^{0} = \rho_0, \quad \rho_\tau^{n+1} \in Prox_{\widetilde{\mathcal{J}}}^{\tau}(\rho^{n}_{\tau})
\end{equation}

}
This define a curve \( t \mapsto \rho_\tau(t) \) in the space of probability measures, with \(\rho_\tau(0) = \rho_0\) and 

\begin{equation}\label{eq:def_JKO2}
\rho_\tau(t) = \rho_\tau^{n+1} \text{ for } t \in (n\tau, (n+1)\tau],
\end{equation}
and $n\in[0,N]$ with $N\tau=T$.

\begin{thm}[JKO scheme]\label{thm:JKO}
Let $a\le \rho_{0}\le b$, where $a,b>0$ be an initial condition in $L^1(\Omega)$ and such that $\mathcal{F}_{p}[\rho(0)]<+\infty$ for all $1\le p<+\infty$. {\color{blue} Let $\rho_{\tau}$ the constructed curve~\eqref{eq:def_JKO2} of the JKO scheme.   Then there exists  $M_d,\nu >0$ in the definition of $\widetilde{\mathcal{J}}$  and $\eps_{d}$ depending on the dimension and the final time of existence $T$ such that if $\|\rho_0\|_{L^{\frac{d}{2}}(\Omega)}\le \eps_d$:
\begin{itemize}
    \item $\|\rho_{\tau}^{n}\|_{L^{\frac d2}}^2\le \frac{M_d}{C_0}$ for some $C$ independent of $n$ and $\widetilde{\mathcal J}[\rho_n]=\mathcal J[\rho_n]$ so the artificial term disappears
    \item $\rho_{\tau}$ is bounded uniformly in $\tau$ in $L^{\infty}((0,T)\times\Omega)\cap L^{\infty}(0,T; W^{1,p}(\Omega))\cap L^{2}(0,T; H^{2}(\Omega))$ for all $1\le p<+\infty$
    \item $\rho_{\tau}$ converges to $\rho$ as $\tau\to 0$ strongly in $L^{2}(0,T; H^{2}(\Omega))$, where $\rho$ is a weak solution of the KS system~\eqref{eq:KS}.
\end{itemize}}
\end{thm}

{\color{blue}
The restriction on $\varepsilon_d$ at the final time is only needed to prove the $L^\infty$ bounds along the interpolation curve, and it comes from the boundedness of the domain. The rest of the argument does not depend on this restriction and relies solely on the $L^\infty$ bound itself. We impose this assumption only to provide a simple way to obtain such bounds. We also remark that similar estimates are available without any restriction on the final time, but only locally in time, as shown in~\cite{carrillo2018}. In particular, our result applies in this local-in-time setting as well.}

\begin{remark}
\begin{enumerate}
{\color{blue}\item On a bounded domain, assuming that  the $L^{\frac d2}$ norm is small, implies that the mass is small, therefore we are treating functions which are not probability measures in the JKO scheme. This would require a small adaptation of the language and the definition of the Wasserstein distance, that we omit here for simplicity. 
\item The $L^{\infty}((0,T)\times\Omega)$ estimate on $\rho_{\tau}$ has been proven in~\cite{dimarino} in dimension $d=2$ and for an initial condition with small mass. The adaptation to the higher dimensional case presents several difficulties: even proving existence of the sequence of the JKO scheme is not immediate and we have to rely on the penalization scheme trick. 
}
\end{enumerate}
\end{remark}

\textbf{Structure of the paper.}
Section~\ref{sect:functional_sobolev} is dedicated to the proof of Theorem~\ref{thm:functional_inequality} and Theorem~\ref{thm:sobolev_bounds}. In particular, we show how Theorem~\ref{thm:functional_inequality} follows from an interpolation between $L^{\infty}$ and $C^{0,\alpha}$ and the Calderón-Zygmund theory. This theorem allows us to prove Theorem~\ref{thm:sobolev_bounds} by computing the dissipation of the functional $\mathcal{F}_{p}$. In Section~\ref{sect:JKO} we prove Theorem~\ref{thm:JKO}. We start by reviewing some concepts of optimal transport theory and functional analysis. We also highlight the relevant literature on the JKO scheme applied to the Keller-Segel system and the Fokker-Planck equation, as our approaches share some similarities. Then, we establish $L^{\infty}_{t}W^{1,p}_{x}$ estimates for the JKO scheme and prove a maximum principle. Finally, we prove the convergence of the scheme in $L^{2}_{t} H^{2}_{x}$.\\

\textbf{Notations.} We denote by $L^{p}(\Omega)$, $W^{m,p}(\Omega)$ the usual Lebesgue and Sobolev spaces, and by $\|\cdot\|_{L^{p}}$, $\|\cdot\|_{W^{m,p}}$ their corresponding norms. As usual, $H^{s}(\Omega) = W^{s,2}(\Omega)$. The spaces $C^{0,\alpha}(\bar{\Omega})$ with $0<\alpha<1$ are the usual Hölder spaces, with corresponding norm $\|\cdot\|_{C^{0,\alpha}}$. We make an abuse of notation when the norms concern a vector, or a matrix and still write for instance $\|\nabla u\|_{L^{p}}$ or $\|D^{2} u\|_{L^{p}}$ instead of $\|\nabla u\|_{L^{p}(\Omega; \R^d)}$ and $\|D^2 u\|_{L^{p}(\Omega;\R^{d^2})}$. We often write $C$ for a generic constant appearing in the different inequalities. Its value can change from one line to another, and its dependence to other constants can be specified by writing $C(a,...)$ if it depends on the parameter $a$ and other parameters.

\section{Functional inequality and Sobolev estimates}\label{sect:functional_sobolev}

The idea behind Theorem~\ref{thm:functional_inequality} is based on the limits of the Calderón-Zygmund estimate for the Hessian, $\|D^2 u \|_{L^p} \leq C(p)\|\Delta u\|_{L^p}$. This inequality is valid for $1 < p < +\infty$, but breaks down at the boundary cases $p = 1$ and $p = +\infty$, where the constant blows up, unless the dimension is one, where the result is obvious as the Hessian and the Laplacian coincide.  In two or greater  dimensions, the result is in general false. As an example let 
$$
u(x,y) = xy \log(x+y), \quad x,y>0. 
$$
We compute:
$$
\Delta u = 2 -\f{2 x y}{(x+y)^2}, \quad \p_{xy}u = \log(x+y) + 1 - \f{xy}{(x+y)^2}.                                       
$$
We observe that $\Delta u \in L^{\infty}$ but not $D^{2} u$. In spite of this, the $L^{\infty}$ norm of a function can be interpolated between its $L^p$ norm (for $p<+\infty)$ and its Hölder continuity norm, $C^{0,\alpha}$. In these two norms, the Hessian can be controlled by the Laplacian. This observation suggests that obtaining a precise value for the constant $C(p)$ in the Calderón-Zygmund estimate could improve these interpolation results and prove our theorem. To achieve this, we first prove how the $L^{\infty}$ norm can be interpolated between the $L^p$ and $C^{0,\alpha}$ norms. Then we evaluate the constant $C(p)$. Lastly, we optimize and prove Theorem~\ref{thm:functional_inequality}.

\begin{lemma}[Interpolation between $L^p$ and $C^{0,\alpha}$] \label{lem:interpolation_Lp_Calpha}
Let $0<\alpha<1$. Let $\Omega$ be a smooth bounded domain in $\R^d$.{\color{blue} In particular, there exists $R_{\Omega}$ and $C_{\Omega}$ such that for all $x\in \overline{\Omega}$ and all $R\le R_\Omega$: $|B(x,R)\cap \Omega|\ge C_{\Omega}|B(x,R)|$.} Then there exists a constant $C$ independent of $p$ such that for all $f\in L^{p}(\Omega)\cap C^{0,\alpha}(\bar{\Omega})$
$$
\|f\|_{L^{\infty}}^{p+\f{d}{\alpha}}\le C \|f\|_{L^{p}
}^{p}\|f\|_{C^{0,\alpha}}^{\f{d}{\alpha}}$$.
\end{lemma}

{\color{blue}
\begin{proof}
Let $x$ be a maximum of $|f|$ in $\overline{\Omega}$. Then:
$$
|f(y)|\ge \|f\|_{L^{\infty}}-\|f\|_{C^{0,\alpha}}R^{\alpha}, \quad \forall y\in B(x,R)\cap\Omega.
$$
We impose the condition 
\begin{equation}\label{eq:cond_R}
R\le \min\left(\left(\frac{\|f\|_{L^{\infty}}}{\|f\|_{C^{0,\alpha}}}\right)^{1/\alpha},R_{\Omega}\right) 
\end{equation}
so that the right hand side is nonnegative.
Then we integrate on $B(x,R)\cap\Omega$:
$$
\int_{B(x,R)\cap\Omega}|f|^{p}\ge  (\|f\|_{L^{\infty}}-\|f\|_{C^{0,\alpha}}R^{\alpha})^{p} |B(x,R)\cap\Omega|\ge C (\|f\|_{L^{\infty}}-\|f\|_{C^{0,\alpha}}R^{\alpha})^{p} R^d
$$
as $|B(x,R)\cap \Omega|\ge C_{\Omega}|B(x,R)|$.
The left hand side is bounded by the $L^p$ norm of $f$ and we obtain 
$$
\|f\|_{L^{\infty}}\le C\f{\|f\|_{L^{p}}}{R^{d/p}}+ \|f\|_{C^{0,\alpha}}R^{\alpha}.
$$
We choose $R^{\alpha+d/p}=\lambda\f{\|f\|_{L^{p}}}{\|f\|_{C^{0,\alpha}}}$  where $\lambda> 0$ is small enough such that~\eqref{eq:cond_R} is satisfied. This yields the result.
\end{proof}}

\begin{lemma}[Calderon-Zygmund estimate]\label{lem:Calderon_Zygmund} Let $\Omega$ be a smooth bounded domain in $\R^d$. 
Let $1< p< 2$, then for all $u\in W^{2,p}(\Omega)$ {\color{blue}with $u=0$ on $\partial\Omega$:}
$$
\|D^{2} u\|_{L^{p}}\lesssim \left(\f{2p}{(p-1)(2-p)}\right)^{1/p} \|\Delta u \|_{L^{p}}.
$$
If $p> 2$, 
$$
\|D^{2} u\|_{L^{p}}\lesssim \left(\f{p(p-1)}{p-2}\right)^{1-1/p}\|\Delta u \|_{L^{p}}.
$$
As a corollary, for $p$ strictly greater than 2, there exists $C$ such that the norm evolves as:
$$
\|D^{2} u\|_{L^{p}}\le Cp\|\Delta u \|_{L^{p}}.
$$
\end{lemma}

\begin{proof}
The proof is based on a careful inspection of the proof of~\cite[Theorem 9.9]{gilbarg2001elliptic}. Let $Tf=\p_{ij}u$, where $\Delta u =f$. The authors prove first a weak $L^1$ and $L^2$ estimate on $T$. By Marcinkiewicz interpolation theorem~\cite[Theorem 9.8]{gilbarg2001elliptic}, they deduce for $1<p\le 2$:
$$
\|Tf\|_{L^{p}}\le C(p)\|f\|_{L^{p}}.
$$
where $C(p)\lesssim  \left(\f{2p}{(p-1)(2-p)}\right)^{1/p}$. By duality (since $T^{*}= T$) let $p>2$ and $q$ the dual exponent $q=\f{p}{p-1}$. Then 
$$
\|Tf\|_{L^{p}}\le C(q)\|f\|_{L^{p}}.
$$
We compute the constant:
$$
C(q) = C\left(\f{p}{p-1}\right) \lesssim\left(\f{p(p-1)}{p-2}\right)^{1-1/p}.
$$
This ends the proof
\end{proof}

{\color{blue}Finally, we recall a classical result, on the Hölder regularity of the Poisson equation (see \cite[Theorem 6.8]{gilbarg2001elliptic} and its proof).
\begin{lemma}\label{lem:Hölder regularity}.
Let $\Omega$ be a smooth bounded domain in $\R^d$. Then there exists $C=C(d,\Omega,\alpha)$ such that for all $u$ satisfying $-\Delta u = \rho$ with $u=0$ on $\p\Omega$:
$$
\|D^{2} u\|_{C^{0,\alpha}} \le C\|\rho\|_{C^{0,\alpha}}.   
$$
\end{lemma}
}
We are now in position to prove Theorem~\ref{thm:functional_inequality}.
\begin{proof}[Proof of the Theorem~\ref{thm:functional_inequality}]
We write the interpolation inequality Lemma~\ref{lem:interpolation_Lp_Calpha} as 
$$
\|f\|_{L^{\infty}}\le C \|f\|_{L^{p}}^{1-\theta}\|f\|_{C^{0,\alpha}}^{\theta}
$$
where $\theta = \theta(p) = \f{d}{\alpha p+d}$.
We choose $f=\p_{ij} u$, where $-\Delta u =\rho$ in $\Omega$ and $u=0$ on $\p\Omega$ and then sum for all $i,j$. Using Lemma~\ref{lem:Calderon_Zygmund} and Lemma~\ref{lem:Hölder regularity} we obtain for $p> 2$:
\begin{align*}
\|D^{2}u\|_{L^{\infty}}\le C p^{1-\theta}\|\rho\|_{L^{p}}^{1-\theta}\|\rho\|_{C^{0,\alpha}}^{\theta}
\le C \f{1}{\theta} \|\rho\|_{L^{p}}^{1-\theta}\|\rho\|_{C^{0,\alpha}}^{\theta}.
\end{align*}
where $-\Delta u =\rho$. The last estimate is simply a consequence of $p^{1-\theta}\le p$ since $p\ge 1$ and $p\le \f{d}{\alpha\theta}$. Before optimizing with respect to $\theta$ we estimate the $L^p$ norm by the $L^{\infty}$ norm 
$$
\|D^{2}u\|_{L^{\infty}}\le C \f{1}{\theta} \|\rho\|_{L^{\infty}}^{1-\theta}\|\rho\|_{C^{0,\alpha}}^{\theta}.
$$
Now we optimize in $\theta$. For that purpose we consider the logarithmn of this previous quantity and optimize
$$
f(\theta) = \log\left(\f 1\theta\right) + (1-\theta) \log(a) + \theta \log(b)
$$
where $a=\|\rho\|_{L^{\infty}}$, $b=\|\rho\|_{C^{0,\alpha}}$. We remind that $b\ge a$.
Differentiating with respect to $\theta$ yields
$$
f'(\theta)=-\f{1}{\theta}+\log\left(\f{b}{a}\right).
$$
The minimum is therefore reached for 
$$
\theta_{0} = \f{1}{\log\left(\f{b}{a}\right)}.
$$
However in our computations we assumed that $0<\theta<1$ and $p>2$ with $\theta = \theta(p) = \f{d}{\alpha p + d}$. And it is possible that such a $\theta_{0}$ does not fall in this range. This would mean that either $\log\left(\f{b}{a}\right)\le 1$ or $\log\left(\f{b}{a}\right)\le \f{2\alpha}{d}+1$. In both cases $\log\left(\f{b}{a}\right)\le C_1$ and we conclude $b\le C a$ for some universal constant $C$. In that particular case, coming back to the definition of $a$ and $b$ that would mean $\|\Delta u\|_{C^{0,\alpha}}\le C\|\Delta u\|_{L^{p}}$ and we could deduce
$$
\|D^{2} u\|_{L^{\infty}}\le C\|D^{2} u\|_{C^{0,\alpha}} \le C \|\Delta u\|_{C^{0,\alpha}}\le C \|\Delta u\|_{L^{p}} \le C \|\Delta u\|_{L^{\infty}}
$$
and the theorem would also be proved with an even better estimate. Therefore we can neglect this case and assume that $\theta_{0}$ is in the range $0<\theta_{0}<1$ and $p> 2$. 
Now we can compute
$$
f(\theta_{0}) = \log\left(\log\left(\f{b}{a}\right)\right)+1 + \log(a)
$$
Then the inequality is simply a consequence of 
$$
\|D^{2}u\|_{L^{\infty}}\le C\exp(f(\theta_{0})).
$$

This concludes the proof of Theorem~\ref{thm:functional_inequality}, from which we deduce Corollary~\ref{cor:functional_inequality}.
\end{proof}

With this functional inequality, we are now prepared to find the Sobolev estimates in the Keller-Segel system, that is to prove Theorem~\ref{thm:sobolev_bounds}. We consider the functional
$$
\mathcal{F}_{p}[\rho]= \int_{\Omega}\rho \left|\nabla (\log \rho - u[\rho])\right|^p\diff x.
$$

Proving that this functional remains bounded along the solutions of the Keller-Segel system, yields $L^{\infty}(0,T; W^{1,p}(\Omega))$ estimates on the solution. Moreover, its dissipation, as it will be clear from the following computations, provides $L^{2}(0,T; H^{2}(\Omega))$ estimates on the solutions. In what follows, we remind that $\vec{n} = (n_{i})_{i}$ is the unit normal to the domain $\Omega$.  

\begin{lemma}[Entropy equality]
Let $\rho$ {\color{blue}be a smooth positive solution} of~\eqref{eq:KS}. With  $Z=\log\rho -u[\rho]$, we have
\begin{equation}
\begin{split}
&\f{d}{dt}\mathcal{F}_{p}[\rho] + p\int_{\Omega} \rho|\nabla Z|^{p-2}(\p_{ij}Z)^2 \diff x + p(p-2)\int_{\Omega} \rho|\nabla Z|^{p-4} |\nabla Z\cdot\nabla\p_i Z|^2\diff x\\=&p\int_{\Omega} \rho |\nabla Z|^{p-2}\nabla Z\cdot D^2 u[\rho]\nabla Z \diff x- p\int_{\Omega} \rho |\nabla Z|^{p-2} \nabla Z\cdot \nabla u[\p_{t}\rho]\diff x\\
&+ p\int_{\p\Omega}\rho \p_{ij}Z \p_{i}Z |\nabla Z|^{p-2} n_{j}\diff \mathcal{H}^{d-1}.
\end{split}
\end{equation} 
\end{lemma}

Here $u[\p_{t}\rho]$ is the solution of the Poisson equation $-\Delta u = \p_t\rho$ with Dirichlet boundary conditions. The notations $\p_{ij}Z,\,  \p_{i}Z, \, \p_{ij}Z\p_{i}Z$ appearing in the integrals simply mean the sum over the indices $i,j$. This notation is used throughout the proof. 

\begin{proof}

\begin{align*}
\f{d}{dt}\int_{\Omega} |\nabla Z|^p\diff \rho &= \int_{\Omega} \p_t\rho |\nabla Z|^p \diff x + p \int_{\Omega} |\nabla Z|^{p-2}\nabla Z\cdot(\nabla \p_t\rho-\p_t\rho\nabla\log\rho)\diff x\\
&-p\int_{\Omega} \rho |\nabla Z|^{p-2} \nabla Z\cdot \nabla u[\p_{t}\rho] \diff x  \\
&= \int_{\Omega} \p_{t}\rho \left(|\nabla Z|^p -p|\nabla Z|^{p-2}\nabla Z\cdot\nabla\log\rho\right) \diff x+ p\int_{\Omega} |\nabla Z|^{p-2}\nabla Z\cdot\nabla\p_t\rho\diff x\\
&- p\int_{\Omega} \rho |\nabla Z|^{p-2} \nabla Z\cdot \nabla u[\p_{t}\rho]\diff x.    \\
&= A+B+C.
\end{align*}
We use the fact that $\p_t \rho = \Delta \rho - \Div(\rho\nabla u[\rho])= \Div(\rho\nabla Z)$.
\begin{align*}
A= &-p\int_{\Omega} \rho \p_i Z |\nabla Z|^{p-2}\nabla Z\cdot \nabla\p_{i} Z \diff x\\
&+p(p-2)\int_{\Omega} \rho \p_i Z |\nabla Z|^{p-4} (\nabla Z\cdot \p_i \nabla Z) (\nabla Z\cdot \nabla\log\rho)\diff x\\
& + p\int_{\Omega} \rho \p_i Z |\nabla Z|^{p-2}\nabla \p_i Z\cdot \nabla\log\rho\diff x\\
& + p\int_{\Omega} \rho \p_i Z |\nabla Z|^{p-2} \nabla Z\cdot \nabla \p_i \log\rho\diff x\\
& + \int_{\p\Omega}\rho \p_{i}Z \left(|\nabla Z|^{p} - p |\nabla Z|^{p-2} \nabla Z\cdot \nabla \log \rho \right)n_{i}\diff \mathcal{H}^{d-1}.
\end{align*}

The boundary conditions imply that the boundary term vanishes. By definition of $Z$ we can rewrite the third term of $A$ as:
$$
p\int_{\Omega} \rho \p_i Z |\nabla Z|^{p-2}\nabla \p_i Z \cdot(\nabla Z+ \nabla u[\rho]) \diff x.
$$

Combining it with the first term of $A$ we obtain 

\begin{align*}
A=&p(p-2)\int_{\Omega} \rho \p_i Z |\nabla Z|^{p-4} (\nabla Z\cdot \p_i \nabla Z) (\nabla Z\cdot \nabla\log\rho)\diff x\\
& + p\int_{\Omega} \rho \p_i Z |\nabla Z|^{p-2}\nabla \p_i Z\cdot \nabla u[\rho]\diff x\\
& + p\int_{\Omega} \rho \p_i Z |\nabla Z|^{p-2} \nabla Z\cdot \nabla \p_i \log\rho\diff x\\
& = A_1 + A_2 + A_3.
\end{align*}

Now we focus on $B$, using once again the formula $\p_t \rho = \Div(\rho\nabla Z)$. 
\begin{align*}
B = &p \int_{\Omega} |\nabla Z|^{p-2}\nabla Z\cdot \nabla \Div(\rho\nabla Z)\diff x\\
= &p \int_{\Omega} |\nabla Z|^{p-2}\p_i Z\p_{ij}(\rho\p_{j}Z)\diff x\\
=& -p(p-2)\int_{\Omega} |\nabla Z|^{p-4}(\nabla Z\cdot\nabla\p_{j}Z)\p_i Z \p_{i}(\rho \p_j Z)\diff x -p\int_{\Omega} |\nabla Z|^{p-2}\p_{ij} Z \p_{i}(\rho\p_j Z)\diff x\\
&+ p\int_{\p\Omega}\p_{i}(\rho \p_{j}Z) \p_{i}Z |\nabla Z|^{p-2} n_{j}\diff \mathcal{H}^{d-1}.
\end{align*}
Expanding the last term in $B$ we obtain 

\begin{align*}
B = & -p(p-2)\int_{\Omega} |\nabla Z|^{p-4}(\nabla Z\cdot\nabla\p_{j}Z)\p_i Z \p_{i}(\rho \p_j Z)\diff x\\
&  -p\int_{\Omega} |\nabla Z|^{p-2}\p_{ij} Z \p_{i}\rho\p_j Z\diff x\\
& -p\int_{\Omega} \rho|\nabla Z|^{p-2}(\p_{ij} Z)^2\diff x\\
& + p\int_{\p\Omega}\p_{i}(\rho \p_{j}Z) \p_{i}Z |\nabla Z|^{p-2} n_{j}\diff \mathcal{H}^{d-1}\\
& = B_1 + B_2 + B_3 + B_4.
\end{align*}
Observe that the last term $B_3$ is a diffusion that has a good sign. The boundary term $B_4$ can be written as
\begin{align*}
B_4 &= p\int_{\p\Omega}\p_{i}(\rho \p_{j}Z) \p_{i}Z |\nabla Z|^{p-2} n_{j}\diff \mathcal{H}^{d-1}\\
& = p\int_{\p\Omega}\p_{i}\rho \p_{j}Z \p_{i}Z |\nabla Z|^{p-2} n_{j} \diff \mathcal{H}^{d-1}+ p\int_{\p\Omega}\rho \p_{ij}Z \p_{i}Z |\nabla Z|^{p-2} n_{j}\diff \mathcal{H}^{d-1}.
\end{align*}
By Neumann boundary condition the first term vanishes and therefore 
$$
B_4 = \int_{\p\Omega}\rho \p_{ij}Z \p_{i}Z |\nabla Z|^{p-2} n_{j}\diff \mathcal{H}^{d-1}.
$$

It remains to compute the sum $A_1+ A_2 + A_3 + B_1 + B_2$.
First we obserte that since $i,j$ are dummy variables:
\begin{align*}
B_1 = &-p(p-2)\int_{\Omega} |\nabla Z|^{p-4}(\nabla Z\cdot\nabla\p_{i}Z)\p_j Z \p_{j}\rho\p_i Z\diff x \\
&-p(p-2)\int_{\Omega} |\nabla Z|^{p-4}(\nabla Z\cdot\nabla\p_{i}Z)\p_j Z \rho\p_{ij} Z\diff x\\
&= -p(p-2)\int_{\Omega} \rho\p_i Z|\nabla Z|^{p-4}(\nabla Z\cdot\nabla\p_{i}Z)(\nabla Z\cdot \nabla \log\rho)\diff x\\
&-p(p-2)\int_{\Omega} |\nabla Z|^{p-4}(\nabla Z\cdot\nabla\p_{i}Z)\p_j Z \rho\p_{ij} Z\diff x\\
&= -p(p-2)\int_{\Omega} \rho\p_i Z|\nabla Z|^{p-4}(\nabla Z\cdot\nabla\p_{i}Z)(\nabla Z\cdot \nabla \log\rho)\diff x \\
&-p(p-2)\int_{\Omega} \rho|\nabla Z|^{p-4}|\nabla Z\cdot\nabla\p_{i}Z|^2\diff x.
\end{align*}

And thus
\begin{align*}
A_1 + B_1 &= -p(p-2)\int_{\Omega} \rho|\nabla Z|^{p-4}|\nabla Z\cdot\nabla\p_{i}Z|^2\diff x,
\end{align*}
which has a good sign {\color{blue}for $p\ge 2$}.

It remains to compute $A_2 + A_3 + B_2$. First observe that
\begin{align*}
B_2 &= -p \int_{\Omega} |\nabla Z|^{p-2}\p_{ij}Z\p_{j}\rho\p_{i}Z\diff x\\
&= -p \int_{\Omega} |\nabla Z|^{p-2}\rho \p_{i}Z (\nabla \p_{i}Z\cdot\nabla \log\rho)\diff x.
\end{align*}

Therefore with the definition of $Z$:
\begin{align*}
A_2 + A_3 + B_2 &= p\int_{\Omega} \rho\p_{i}Z |\nabla Z|^{p-2}\nabla Z\cdot \nabla \p_i u[\rho]\diff x\\
&=p\int_{\Omega} \rho |\nabla Z|^{p-2}\nabla Z\cdot D^2 u[\rho]\nabla Z\diff x.
\end{align*}
This concludes the proof. 
\end{proof}

The boundary term can be rewritten with the following lemma from~\cite[Lemma 2.5]{toshpulatov}.
\begin{lemma}\label{lem:boundary_h}
Suppose $\Omega=\{h<0\}$ for a smooth function $h: \R^d\to \R$ with $\nabla h\ne 0$ on $\{h=0\}$. Therefore $\vec{n}=\f{\nabla h}{|\nabla h|}$ and for all $v$ with $v \cdot n =0$ on $\p\Omega$, we have for all $x\in \p\Omega$ the equality
$$
v(x)\cdot Dv(x) n(x) = -\f{v(x)\cdot D^{2}h(x)v(x)}{|\nabla h(x)|}.
$$
\end{lemma}

We can therefore rewrite the entropy equality in the following form  
\begin{proposition}[Entropy equality]\label{prop:entropy_equality}
Let $\rho$ be a {\color{blue}smooth positive solution} of~\eqref{eq:KS}. Suppose $\Omega=\{h<0\}$ for a smooth function $h: \R^d\to \R$ with $\nabla h\ne 0$ on $\{h=0\}$. With  $Z=\log\rho - u[\rho]$, we have
\begin{equation}\label{eq:ODE_FP}
\begin{split}
&\f{d}{dt}\mathcal{F}_{p}[\rho] + p\int \rho|\nabla Z|^{p-2}(\p_{ij}Z)^2\diff x + p(p-2)\int \rho|\nabla Z|^{p-4} |\nabla Z\cdot\nabla\p_i Z|^2\diff x\\=&p\int \rho |\nabla Z|^{p-2}\nabla Z\cdot D^2 u[\rho]\nabla Z\diff x - p\int \rho |\nabla Z|^{p-2} \nabla Z\cdot \nabla u[\p_{t}\rho]\diff x\\
&- p\int_{\p\Omega}\rho \f{\nabla Z\cdot D^{2} h \nabla Z}{|\nabla h|} |\nabla Z|^{p-2}\diff \mathcal{H}^{d-1}.
\end{split}
\end{equation} 
\end{proposition}

In order to prove that this functional provides a priori estimates on the solution, it is necessary to bound the right-hand side. Here, we assume that we already have $L^{\infty}((0,T)\times\Omega)$ estimates on $\rho$. There are three terms on the right-hand side. The last one is the boundary term and has the good sign since we assume that the domain is convex and therefore $h$ is convex. The first term is the most challenging one to manage. The strategy is to estimate the norm of $D^{2} u$ in $L^{\infty}$ and then write

\begin{equation*} \left|\int_{\Omega} \rho |\nabla Z|^{p-2} \nabla Z\cdot D^2 u[\rho] \nabla Z \diff x\right| \leq \|D^2 u[\rho]\|_{L^{\infty}} \int_{\Omega} \rho |\nabla Z|^{p}\diff x. \end{equation*}

On the right-hand side, we recognize the functional $\mathcal{F}_{p}$, which allows us to apply Gronwall's lemma. The main difficulty, however, lies in the fact that we only have a priori estimates for $\rho$ in $L^{\infty}$, which implies that $D^{2}u$ belongs to BMO but not to $L^{\infty}$. Nevertheless, to use Gronwall's lemma, we can tolerate having $\mathcal{F}_{p}(t) \log \mathcal{F}_{p}(t)$ on the right-hand side. Indeed consider the following differential inequality for the function $y(t)$:
$$
y' \le y \log y.
$$
In the variable $z = \log y$, this inequality becomes
$$
z' \le z,
$$
which implies the estimate \(z(t) \leq z_0 e^{t}\). Therefore, we obtain the estimate $y(t) \leq y_0^{e^{t}}$. This approach is particularly useful in Section~\ref{sect:JKO}, in the discrete setting of the JKO scheme, where the analogous differential inequality is obtained at the discrete level. The logarithmic term appearing on the right-hand side is found  with the previous functional inequality, Theorem~\ref{thm:functional_inequality}.

We are now in position to prove that the right-hand side terms can be controlled.

\begin{proposition}\label{prop:rhs}
Let $\rho$ be  {\color{blue} a smooth positive function and $u$ such that $-\Delta u =\rho$ in $\Omega$ and $u=0$ on $\partial\Omega$}. Let $T_1$ and $T_2$ be the first two terms on the right-hand side of the equation~\eqref{eq:ODE_FP}. Let {\color{blue}$p>d$ and $\alpha<1-\f dp$}. Then there exists $C$, depending on $\|\rho(t)\|_{L^{\infty}},p,d$ such that
\begin{align*}
&|T_1| \le C \mathcal{F}_{p}[\rho](t)(1+\log (\mathcal{F}_{p}[\rho](t)+1)),  \\
&|T_2| \le C\mathcal{F}_{p}[\rho](t).
\end{align*}
\end{proposition}

Before proving this proposition, we need the following lemma
\begin{lemma}\label{lem:rho_W1p}
Let $\rho$ be a {\color{blue}smooth positive function and $u$ such that $-\Delta u =\rho$ in $\Omega$ and $u=0$ on $\partial\Omega$}. Then
$$
\|\rho(t)\|_{W^{1,p}}\le C (\mathcal{F}_{p}[\rho](t)^{1/p} +1),
$$
where the constant $C$ depends on $\|\rho(t)\|_{L^{\infty}}$.
\end{lemma}

\begin{proof}
We write 
\begin{align*}
\int_{\Omega} |\nabla \rho|^{p}\diff x&= \int_{\Omega} \left|\f{\nabla\rho}{\rho}\right|^{p}\rho^{p-1}\rho\diff x\\
&\le \|\rho(t)\|_{L^{\infty}}^{p-1}\int_{\Omega} \left|\f{\nabla\rho}{\rho}\right|^{p}\rho \diff x\\
& \le 2^p\|\rho(t)\|_{L^{\infty}}^{p-1}\int_{\Omega} \left|\f{\nabla\rho}{\rho} - \nabla u[\rho]\right|^{p}\rho \diff x+ 2^p\|\rho(t)\|_{L^{\infty}}^{p-1} \int_{\Omega} |\nabla u[\rho]|^{p}\rho\diff x\\
&\le 2^p\|\rho(t)\|_{L^{\infty}}^{p-1} \mathcal{F}_{p}(t) + 2^p \|\rho(t)\|_{L^{\infty}}^{p} \|\nabla u[\rho]\|_{L^{p}}^{p}.
\end{align*}
From the first to the second line, we estimated $\rho(t)^{p-1}$ by its $L^{\infty}$ norm. In the third line we used $|x+y|^p \le 2^{p}(|x|^{p} + |y|^{p})$. In the fourth line we have estimated $\rho(t)$ by its $L^{\infty}$ norm.
Next we observe  that by Poincaré inequality and the Calderón-Zygmund estimate $\|D^{2}u[\rho]\|_{L^{p}}\le C\|\Delta u[\rho]\|_{L^{p}}$: 
$$
\|\nabla u[\rho]\|_{L^{p}}\le C\|D^{2}u[\rho]\|_{L^{p}} \le C\|\Delta u[\rho]\|_{L^{p}} = C\|\rho\|_{L^{p}}\le C \|\rho\|_{L^{\infty}}.
$$
{\color{blue}The first inequality follows by Poincaré-Wirtinger inequality, using Dirichlet boundary conditions.}
Therefore 

\begin{align*}
\|\rho(t)\|_{W^{1,p}} &= \|\rho(t)\|_{L^p} + \|\nabla \rho(t)\|_{L^{p}}\\
&\le \|\rho(t)\|_{L^\infty} + 2C \|\rho(t)\|_{L^\infty}^{(p+1)/p} + 2C \|\rho(t)\|_{L^\infty}^{(p-1)/p}\mathcal{F}_{p}(t)^{1/p}.
\end{align*}
\end{proof}

\begin{proof}[Proof of the Proposition~\ref{prop:rhs}]
We start with $T_1$. 
$$
|T_1| \le p \|D^{2} u[\rho]\|_{L^{\infty}}\int_{\Omega} \rho |\nabla Z|^p\diff x = p\|D^{2} u[\rho]\|_{L^{\infty}}\mathcal{F}_{p}[\rho] . 
$$
Now using Corollary~\ref{cor:functional_inequality}, the fact that $\|\rho\|_{C^{0,\alpha}}\le C \|\rho\|_{W^{1,p}}$  and Lemma~\ref{lem:rho_W1p} we find
\begin{align*}
\|D^{2} u[\rho]\|_{L^{\infty}}&\le C(1+\|\rho\|_{L^{\infty}}( 1 + \log^{+}(C(\mathcal{F}_{p}[\rho]+1)))\\
&\le C ( 1+ \log(\mathcal{F}_{p}[\rho]+1))
\end{align*}
with a new constant $C$ depending on $\|\rho\|_{L^{\infty}}$. We can conclude to the estimate on $T_1$ as announced.  We turn our attention on $T_2$.  By Hölder's inequality:
\begin{align*}
|T_2| &\le p\left(\int_{\Omega} \rho^{p/(p-1)}|\nabla Z|^p \diff x\right)^{1-1/p}\left(\int_{\Omega} |\nabla u[\p_t\rho]|^{p}\diff x\right)^{1/p}\\
&\le p\|\rho\|^{(p-1)/p^2}_{L^{\infty}} \mathcal{F}_{p}(t)^{1-1/p}\left(\int_{\Omega} |\nabla u[ \p_t\rho]|^{p}\diff x\right)^{1/p},
\end{align*}
where we have used $|\rho^{p/(p-1)}|\le \rho \|\rho\|^{1/p}_{L^{\infty}}$. The properties of the Keller-Segel kernel, the Calderón-Zygmund estimate Lemma~\ref{lem:Calderon_Zygmund} and the fact that $\p_t \rho = \Div(\rho\nabla Z)$ yield:
$$
\|\nabla u[\p_t \rho]\|_{L^{p}}\le C\|\rho\nabla Z\|_{L^{p}}\le C \|\rho\|^{1-1/p}_{L^{\infty}}\mathcal{F}_{p}(t)^{1/p}. 
$$
In the last line we used $|\rho|\le \rho^{1/p} \|\rho\|^{1-1/p}_{L^{\infty}}$. This ends the proof. 
\end{proof}

\begin{proof}[Proof of Theorem~\ref{thm:sobolev_bounds}]
{\color{blue}All of the computations can be carried rigorously in an approximation scheme as in~\cite{SUGIYAMA2006333}. We first find an $L^{\infty}$ bound by the Alikakos iteration method, as it is done in the JKO scheme below. At the continuous level, the proof is in fact simpler. Positivity follows by maximum principle. Then the theorem is a consequence of Propositions~\ref{prop:entropy_equality} and \ref{prop:rhs} as well as Gronwall's lemma.}.

\end{proof}

\section{Estimates at the discrete level: JKO scheme}\label{sect:JKO}

In this section, we prove Theorem~\ref{thm:JKO}. We first introduce few mathematical tools from optimal transport theory and the JKO scheme. The reader who is familiar with these topics may choose to skip the first subsection. For more details on these ideas, we refer to the books~\cite{santambrogio2015optimal,villani2021topics,villani2009optimal,ambrosio2008gradient}. In a second subsection, we review the main results from the literature concerning the JKO scheme in the context of the Keller-Segel model, as established in~\cite{dimarino} and state general lemmas. The last two subsections are dedicated to the proof of Theorem~\ref{thm:JKO}: first existence of the scheme and $L^{\infty}$ bounds, second the $L^{\infty}_{t}W^{1,p}_{x}$ estimates and finally the $L^{2}_{t}H^{2}_{x}$ strong convergence.

\subsection{Preliminaries on gradient flows and the JKO scheme}

 A gradient flow describes how a system evolves by following the steepest descent of a given functional. They can be formulated in metric spaces with the Wasserstein distance $W_2$ defined as

$$
W_{2}^2(\rho,\eta) = \inf_{T:\Omega\to\Omega}\left\{\int_{\Omega}|x-T(x)|^2\diff \rho, \, T\#\rho=\eta\right\}.
$$

More generally we define the $p$-Wasserstein distance as

$$
W_{p}^p(\rho,\eta) = \inf_{T:\Omega\to\Omega}\left\{\int_{\Omega}|x-T(x)|^p\diff \rho, \, T\#\rho=\eta\right\}.
$$

The inf is a problem introduced by Monge in~\cite{monge1781memoire} and reformulated by Kantorovich in~\cite{kantorovich1942transfer} in a convex form. Kantorovich also provided a dual formulation for the Wasserstein distance
$$
\f{1}{2}W_{2}^2(\rho,\eta) = \sup_{\varphi(x)+\psi(y)\le \f{1}{2}|x-y|^2}\left\{\int_{\Omega}\varphi\diff \rho +\int_{\Omega}\psi \diff \eta\right\}.
$$

The optimal $\varphi$ is called the Kantorovich potential, it is Lipschitz continuous and $\f{|x|^2}{2}-\varphi$ is a convex function as a result of the Brenier theorem~\cite{brenier1987decomposition,brenier1991polar}. It is linked to the optimal transport map $T$ by $T(x) = x - \nabla\varphi(x)$. In particular, using the dual formulation, one can deduce that the first variation of the Wasserstein distance with respect to $\rho$ is $\varphi$.

The general form of the gradient flow PDE for a functional $F:\mathcal{P}(\Omega)\to \R\cup\{+\infty\}$ in Wasserstein space is given by:

$$
\begin{cases}
&\partial_t \rho - \Div  \left( \rho \nabla \frac{\delta F}{\delta \rho} \right) = 0,\\
&\rho(0,\cdot) = \rho_{0},
\end{cases}
$$

where $\frac{\delta F}{\delta \rho}$ is the first variation of the functional $F$ with respect to the probability measure $\rho$. The corresponding minimizing movement scheme is defined as a sequence $\{\rho_\tau^n\}_{n}$ such that $\rho_{\tau}^{0}= \rho_{0}$ and

\[
\rho_\tau^{n+1} \in \arg\min_{\rho} \left\{ F(\rho) + \frac{W_2^2(\rho, \rho_\tau^n)}{2\tau} \right\}, 
\]

that we write $\rho^{n+1}_{\tau}\in Prox_{F}^{\tau}[\rho^{n}_{\tau}]$. This sequence can then be used to define a curve \( t \mapsto \rho_\tau(t) \) in the space of probability measures, with \(\rho_\tau(0) = \rho_0\) and 

\[
\rho_\tau(t) = \rho_\tau^{n+1} \text{ for } t \in (n\tau, (n+1)\tau].
\]

Under suitable conditions on $F$, the scheme converges to a solution of the corresponding gradient flows. This scheme, known as the JKO scheme, has first been used in the context of the Fokker-Planck equation~\cite{MR1617171}, and is now a common approach to prove the existence of weak solutions to a PDE. In this setting it is possible to prove that the optimal condition on the measure $\rho_{\tau}^{n+1}$ is 
$$
\f{\varphi}{\tau} +\frac{\delta F}{\delta \rho}(\rho^{n+1}_{\tau}) = C \quad a.e. \text{ on $\{\rho^{n+1}_{\tau}>0\}$},
$$
where $C$ is a constant and $\varphi$ is the Kantorovich potential in the transport from $\rho^{n+1}_{\tau}$ to $\rho^{n}_{\tau}$. $\rho^{n+1}_{\tau}$ is also linked to $\rho^{n}_{\tau}$ by the Monge-Ampère equation
$$
\rho^{n+1}_{\tau}(x) = \rho^{n}_{\tau}(x-\nabla\varphi(x)) \det(I - D^{2}\varphi(x)). 
$$

\subsection{The JKO scheme in the Keller-Segel system and useful lemmas}

The JKO scheme for the Keller-Segel system has been studied in~\cite{blanchet-ks-gradient-2013,blanchet2013gradient,blanchet2008convergence}. More estimates of the scheme and in particular $L^{\infty}$, and Sobolev estimates were studied in~\cite{dimarino} in dimension 2. {\color{blue}In higher dimensions, even proving existence of the minimizers of the JKO scheme proves to be difficult. In dimension 2, the proof uses the Logarithmic Hardy-Littlewood-Sobolev inequality which states that we can control the entropy $\rho\mapsto\int\rho\log\rho$ from the bound on the free energy (that is the aggregation term is smaller than the entropy). From this estimate we deduce weak compactness in $L^1$ of the minimizing sequence which is enough. In higher dimensions, such an inequality does not hold anymore as the aggregation term (coming from $-\Delta u =\rho$) is not the same anymore. In fact one would need to have a priori $L^p$ bounds for some $p$ on the solution. This is obtained by introducing the penalizing scheme~\eqref{eq:def_JKO}. 

\begin{proposition}\label{prop:existence_JKO}
 Let $\Omega\subset\R^d$.   Then the sequence $\rho^{n+1}_{\tau}\in Prox_{\widetilde{\mathcal{J}}}^{\tau}(\rho^{n}_{\tau})$ starting from $\rho_0$ is well defined, with $\rho_{0}$ as in Theorem~\ref{thm:JKO}.   
\end{proposition}

In order to prove Proposition~\ref{prop:existence_JKO} we need the following lemma written in~\cite[Lemma 5.10]{dimarino}, and taken from~\cite[Theorem 10.15]{giusti2003direct}.
\begin{lemma}\label{lem:sobone2} Let $\Omega$ be a bounded convex domain, and $f \in{(W^{1,p}_0)^*(\Omega)}$ be given. Let $u$ the unique solution in $W^{1,q}_0(\Omega)$ (with $\f{1}{p} + \f{1}{q}=1$ the dual exponent to $p$) of $-\Delta u=f$. Then there exists a constant $C>0$, depending only on the dimension, on $p$ and possibly on $\Omega$, such that
$$  \| \nabla u \|_{q} \leq  C\| f \|_{(W^{1,p})^*(\Omega)} . $$
\end{lemma}

\begin{proof}[Proof of Proposition~\ref{prop:existence_JKO}]
For $d=2$ the claim is classical, see e.g.~\cite{dimarino}. We therefore assume throughout that
$d\ge 3$, and we fix $n\in\mathbb{N}$ and $\tau>0$.
Recall that the JKO step is the minimization problem
\[
\rho \ \longmapsto\ \frac1{2\tau}W_2^2(\rho,\rho_\tau^n)\;+\;\widetilde{\mathcal{J}}[\rho].
\]

\medskip

\underline{Step 1: The aggregation term.}
Integrating by parts gives the identity
\[
\int_\Omega |\nabla u|^2\,dx  =  \int_\Omega \rho\,u\,dx.
\]
Using~\eqref{eq:estimate_H1_Ld_rewritten}, that uses Lemma~\ref{lem:sobone2}, $\Omega$ is bounded and $\frac d2\ge \frac{2d}{d+2}$ we obtain:
\begin{equation}\label{eq:estimate_H1_Ld_rewritten2}
\int_\Omega \rho\,u \le  C_0\,\|\rho\|_{L^{\frac d2}(\Omega)}^2,
\end{equation}

Estimate \eqref{eq:estimate_H1_Ld_rewritten2} is the key point: the interaction term can be
controlled provided one has an a priori $L^{d/2}$ bound on $\rho$.
Such a bound is expected to be
propagated (and in fact nonincreasing) in the Keller-Segel system~\cite{MR2099126}. We therefore enforced it
in the penalized minimization which, as we will  see later, is just artificial and not seen in the scheme.

\medskip

\underline{Step 2: Uniform bounds for a minimizing sequence.}
Let $(\rho_k)_k$ be a minimizing sequence for the above problem. Since $\rho_\tau^n$ is an admissible competitor, we have
\[
\frac1{2\tau}W_2^2(\rho_k,\rho_\tau^n) + \widetilde{\mathcal{J}}[\rho_k]
 \le 
\widetilde{\mathcal{J}}[\rho_\tau^n] 
\]
In particular, $(\widetilde{\mathcal{J}}[\rho_k])_k$ is bounded from above.

We claim that $(\rho_k)_k$ is bounded in $L^{d/2}(\Omega)$. Indeed:
\begin{itemize}
\item If $C_0\|\rho_k\|_{L^{\frac d2}}^2\le M_d$, then 
\begin{equation}\label{eq:JKO_Ld2_1}
 \|\rho_k\|_{L^{\frac d2}}^2\le M_d/C_0   .
\end{equation}
\item If $C_0\|\rho_k\|_{L^{\frac d2}}^2> M_d$, then the penalization is active and
\[
\widetilde{\mathcal{J}}[\rho_k]
=
\mathcal{J}[\rho_k] + \frac{1}{\nu}\left(C_0\|\rho_k\|_{L^{\frac d2}}^2 - M_d\right).
\]
Using the definition of $\mathcal{J}$ and \eqref{eq:estimate_H1_Ld_rewritten}, we get
\[
\widetilde{\mathcal{J}}[\rho_k]
 \ge 
\int_\Omega \rho_k\log\rho_k
 - \frac12\int_\Omega \rho_k u_k
 + \frac{1}{\nu}\left(C_0\|\rho_k\|_{L^{\frac d2}}^2 - M_d\right)
 \ge 
\int_\Omega \rho_k\log\rho_k
 + \left(\frac 1\nu-\frac{C_0}{2}\right)\|\rho_k\|_{L^{\frac d2}}^2 - \frac 1\nu M_d,
\]
where $u_k=u_{\rho_k}$. We can assume that $\nu\le \frac{1}{C_0}$ so that $\frac 1\nu-\frac{C_0}{2}\ge\frac{1}{2\nu}$.  Since $\int_\Omega \rho_k\log\rho_k$ is bounded below on a bounded domain, the upper bound on $\widetilde{\mathcal{J}}[\rho_k]$ implies a uniform bound 
\begin{equation}\label{eq:JKO_Ld2_2}
\|\rho_k\|_{L^{\frac d2}}^2\le C\nu+ C\nu \widetilde{\mathcal{J}}[\rho_\tau^n]+ M_d.
\end{equation}
\end{itemize}
Therefore in all cases $(\rho_k)_k$ is uniformly bounded in $L^{d/2}(\Omega)$.

Moreover, again by \eqref{eq:estimate_H1_Ld_rewritten} and the uniform $L^{d/2}$ bound, the
interaction term $\int\rho_k u_k$ is uniformly bounded. Since $\widetilde{\mathcal{J}}[\rho_k]$
is bounded from above, we also obtain a uniform upper bound on the entropy
$\int_\Omega \rho_k\log\rho_k$ (up to an additive constant depending on $M$ and $\Omega$).
In particular, $(\rho_k)_k$ is uniformly integrable.

\medskip

\underline{Step 3: Compactness and convergence.}
Up to extracting a subsequence, we may assume
\[
\rho_k \rightharpoonup \rho \quad\text{weakly in }L^{\frac d2}(\Omega),
\qquad\text{and}\qquad
\rho_k \to \rho \quad\text{narrowly as measures.}
\]
Let $u_k$ solve $-\Delta u_k=\rho_k$ with $u_k=0$ on $\partial\Omega$.
From \eqref{eq:estimate_H1_Ld_rewritten} and the uniform $L^{d/2}$ bound, we get a uniform bound
on $\|\nabla u_k\|_{L^2}$, hence on $\|u_k\|_{H_0^1(\Omega)}$.
By Rellich--Kondrachov,
\[
u_k \to u \quad\text{strongly in }L^q(\Omega)\quad\text{for every }q<\frac{2d}{d-2},
\]
in particular for $q=\frac{d}{d-2}$.
Passing to the limit in the weak formulation shows that $u$ is the Dirichlet solution of
$-\Delta u=\rho$.

Since $\rho_k\rightharpoonup\rho$ in $L^{\frac d2}$ and $u_k\to u$ strongly in
$L^{\frac{d}{d-2}}$, we conclude that
\[
\int_\Omega \rho_k\,u_k  \longrightarrow  \int_\Omega \rho\,u.
\]

\medskip

\underline{Step 4: Lower semicontinuity and existence of a minimizer.}
The map $\rho\mapsto \frac1{2\tau}W_2^2(\rho,\rho_\tau^n)$ is lower semicontinuous
with respect to narrow convergence on the bounded set $\Omega$.
The entropy is lower semicontinuous with respect to $L^1$--weak convergence (and we have uniform
integrability), and the penalty term is lower semicontinuous by convexity of
$\rho\mapsto \|\rho\|_{L^{\frac d2}}^2$.
Together with the convergence of $\int\rho_k u_k$, we can pass to the limit along the minimizing
sequence and obtain that $\rho$ is a minimizer of the penalized problem.
\end{proof}

Before proving the $L^{\infty}$ bound we prove the positivity of the sequence from the JKO scheme. This follows since $\rho\mapsto\int_{\Omega}\rho\log\rho$  has an infinite slope at 0 and therefore is a barrier against vanishing sets of positive Lebesgue measure.

\begin{lemma}[Strict positivity of the JKO scheme]\label{lem:strict_positivy}
Let $\rho^{n+1}_{\tau}\in Prox_{\widetilde{\mathcal{J}}}^{\tau}(\rho^{n}_{\tau})$ be the penalized JKO sequence
associated to the Keller-Segel system, with initial density $\rho_0$ as in Theorem~\ref{thm:JKO}.
Then $\rho^n_\tau>0$ almost everywhere in $\Omega$ for every $n\ge 0$.
\end{lemma}

\begin{proof}
Fix $n\ge 0$ and denote
\[
\mu:=\rho_\tau^n,
\qquad
\rho:=\rho_\tau^{n+1}\in Prox_{\widetilde{\mathcal J}}^\tau(\mu).
\]
That is, $\rho$ minimizes the penalized JKO scheme
\begin{equation}\label{eq:def_JKO_penalized}
\widetilde{\mathcal J}_\tau(\sigma\,|\,\mu)
:=\frac{1}{2\tau}W_2^2(\sigma,\mu)
+\mathcal J[\sigma]
+\frac1\nu\bigl(C_0\|\sigma\|_{L^{d/2}(\Omega)}-M_d\bigr)_+,
\end{equation}
among all densities $\sigma\ge 0$ with $\int_\Omega \sigma\,dx=M$, where $M$ is the mass of $\rho_0$ (and $\mu$).

We argue by contradiction. Assume that the zero set
\[
A:=\{x\in\Omega:\rho(x)=0\}
\]
has positive measure, $|A|>0$.

Let $\eta\equiv M/|\Omega|$ be the constant density with mass $M$.
For $\varepsilon\in(0,1)$ define the convex perturbation
\[
\rho_\varepsilon:=(1-\varepsilon)\rho+\varepsilon\eta.
\]
Then $\rho_\varepsilon\ge \varepsilon\eta>0$ a.e., $\rho_\varepsilon$ has the same mass $M$, and therefore it is an admissible
competitor in~\eqref{eq:def_JKO_penalized}. By minimality of $\rho$,
\begin{equation}\label{eq:minimality_start}
0\le \widetilde{\mathcal J}_\tau(\rho_\varepsilon\,|\,\mu)-\widetilde{\mathcal J}_\tau(\rho\,|\,\mu).
\end{equation}
We show that the right-hand side is strictly negative for $\varepsilon>0$ small enough, yielding a contradiction.

\medskip
\underline{Step 1: Wasserstein term.}
We use convexity of $W_2^2(\,\cdot\,,\mu)$ with respect to the first argument.
Let $\gamma_\rho$ be an optimal transport plan between $\rho$ and $\mu$, and $\gamma_\eta$ an optimal plan between $\eta$ and $\mu$.
Then
\[
\gamma_\varepsilon:=(1-\varepsilon)\gamma_\rho+\varepsilon\gamma_\eta
\]
is a transport plan between $\rho_\varepsilon$ and $\mu$, hence
\begin{align*}
W_2^2(\rho_\varepsilon,\mu)
&\le \int_{\Omega\times\Omega}|x-y|^2\,d\gamma_\varepsilon(x,y)\\
&=(1-\varepsilon)\int |x-y|^2\,d\gamma_\rho+\varepsilon\int |x-y|^2\,d\gamma_\eta\\
&=(1-\varepsilon)W_2^2(\rho,\mu)+\varepsilon W_2^2(\eta,\mu).
\end{align*}
Therefore
\begin{equation}\label{eq:W2_Oeps}
\frac{1}{2\tau}\Bigl(W_2^2(\rho_\varepsilon,\mu)-W_2^2(\rho,\mu)\Bigr)
\le \frac{\varepsilon}{2\tau}\,W_2^2(\eta,\mu)
=C\,\varepsilon,
\end{equation}

\medskip
\underline{Step 2: Aggregation term.}
Denote the Keller--Segel interaction energy by
\[
I(\sigma):=-\frac12\int_\Omega \sigma\,u[\sigma]\,dx,
\]
Expanding,
\begin{align*}
I(\rho_\varepsilon)=(1-\varepsilon)^2 I(\rho)
-\frac{\varepsilon(1-\varepsilon)}2\int_\Omega \rho\,u[\eta]\,dx
-\frac{\varepsilon(1-\varepsilon)}2\int_\Omega \eta\,u[\rho]\,dx
+\varepsilon^2 I(\eta).
\end{align*}
Define the cross term
\[
B(\rho,\eta):=-\frac12\int_\Omega \rho\,u[\eta]\,dx=-\frac12\int_\Omega \eta\,u[\rho]\,dx.
\]
Then we obtain

\begin{align}
I(\rho_\varepsilon)-I(\rho)=2\varepsilon\bigl(B(\rho,\eta)-I(\rho)\bigr)+O(\varepsilon^2).\label{eq:I_diff_Oeps}
\end{align}
Note that~\eqref{eq:JKO_Ld2_1}-\eqref{eq:JKO_Ld2_2} yield $L^{d/2}(\Omega)$ bounds on $\rho$, and $\eta$ is bounded in $L^{d/2}(\Omega)$ by definition. Thus
\[
|I(\rho)|\le C\|\rho\|_{L^{d/2}}^2,\qquad |B(\rho,\eta)|\le C\|\rho\|_{L^{d/2}}\|\eta\|_{L^{d/2}},
\]
so the right-hand side of~\eqref{eq:I_diff_Oeps} is bounded by $C\varepsilon$ for a constant $C$ independent of $\varepsilon$:
\begin{equation}\label{eq:I_Oeps}
I(\rho_\varepsilon)-I(\rho)\le C\varepsilon.
\end{equation}

\medskip
\underline{Step 3: Penalization term.}
Define the penalty functional
\[
P(\sigma):=\frac{1}{\nu}\bigl(C_0\|\sigma\|_{L^{d/2}(\Omega)}-M_d\bigr)_+.
\]
Then
\begin{align}
P(\rho_\varepsilon)-P(\rho)
&\le C\|\rho_\varepsilon-\rho\|_{L^{d/2}}
= C\|\varepsilon(\eta-\rho)\|_{L^{d/2}}
= C\,\varepsilon.
\label{eq:penalty_Oeps}
\end{align}
\medskip
\underline{Step 4: Entropy.}
Let $f(s)=s\log s$ with the convention $f(0)=0$.
We estimate $\int_\Omega (f(\rho_\varepsilon)-f(\rho))$ by splitting $\Omega=A\cup(\Omega\setminus A)$.

\smallskip
On $A$:
Since $\rho=0$ on $A$, we have $\rho_\varepsilon=\varepsilon\eta$ on $A$, hence
\begin{align}
\int_A f(\rho_\varepsilon)\,dx
&=\int_A \varepsilon\eta\log(\varepsilon\eta)\,dx
=\varepsilon\log\varepsilon\int_A \eta\,dx+\varepsilon\int_A \eta\log\eta\,dx.
\label{eq:entropy_on_A}
\end{align}

On $\Omega\setminus A$:
By convexity of $f$ on $[0,\infty)$,
\[
f\bigl((1-\varepsilon)\rho+\varepsilon\eta\bigr)\le (1-\varepsilon)f(\rho)+\varepsilon f(\eta),
\]
hence
\[
f(\rho_\varepsilon)-f(\rho)\le \varepsilon\bigl(f(\eta)-f(\rho)\bigr).
\]
Integrating over $\Omega\setminus A$ gives
\begin{equation}\label{eq:entropy_off_A}
\int_{\Omega\setminus A}\bigl(f(\rho_\varepsilon)-f(\rho)\bigr)\,dx
\le \varepsilon\int_{\Omega\setminus A}\bigl(f(\eta)-f(\rho)\bigr)\,dx
\le C\,\varepsilon,
\end{equation}
the last estimate following from~\eqref{eq:JKO_Ld2_1}-\eqref{eq:JKO_Ld2_2}.

Combining~\eqref{eq:entropy_on_A} and~\eqref{eq:entropy_off_A}, we obtain
\begin{equation}\label{eq:entropy_total}
\int_\Omega \rho_\varepsilon\log\rho_\varepsilon-\rho\log\rho\,dx
\le
\varepsilon\log\varepsilon\int_A \eta\,dx + C\,\varepsilon,
\end{equation}
for some constant $C$ independent of $\varepsilon$.

\medskip
\underline{Step 5: Contradiction.}
Putting everything together, we obtain
\[
0\le
\widetilde{\mathcal J}_\tau(\rho_\varepsilon\,|\,\mu)-\widetilde{\mathcal J}_\tau(\rho\,|\,\mu)
\le
\varepsilon\log\varepsilon\int_A \eta\,dx + C\varepsilon,
\]
for some constant $C$ independent of $\varepsilon$. Since $\eta\equiv M/|\Omega|$, we have
\[
\int_A \eta\,dx=\frac{M}{|\Omega|}|A|>0.
\]
Hence, there exists $\varepsilon_0\in(0,1)$ such that the right-hand side is strictly negative for all
$\varepsilon\in(0,\varepsilon_0)$, contradicting~\eqref{eq:minimality_start}.
Therefore $|A|=0$, i.e.\ $\rho=\rho_\tau^{n+1}>0$ almost everywhere.
\end{proof}

Now we turn our attention on the $L^{\infty}$ bound of the sequence of the JKO scheme. We first prove an $L^{p_0}$ bound for some $p_0>d$. This is enough to use an Alikakos iteration method at the level of the JKO scheme.

\begin{proposition}\label{lem:Linfty_JKO}
 Let  $\rho^{n+1}_{\tau}\in Prox_{\widetilde{\mathcal{J}}}^{\tau}(\rho^{n}_{\tau})$ starting from $\rho_0$ the sequence of the JKO scheme as in Theorem~\ref{thm:JKO}. Then there exists  $M_d,\nu >0$ in the definition of $\widetilde{\mathcal{J}}$  and $\eps_{d}$ depending on the dimension and the final time of existence $T$ such that if $\|\rho_0\|_{L^{\frac{d}{2}}(\Omega)}\le \eps_d$:
\begin{itemize}
    \item $\|\rho_{\tau}^{n}\|_{L^{\frac d2}}^2\le \frac{M_d}{C_0}$ for some $C$ independent of $n$ and $\widetilde{J}[\rho_n]=J[\rho_n]$ so the artificial term disappears
    \item $\|\rho^{n}_{\tau}\|_{L^{\infty}}\le C(\|\rho_{0}\|_{L^{\infty}})$.
\end{itemize}
In particular, the interpolation curve $\rho_{\tau}$ remains bounded in $L^{\infty}((0,T)\times\Omega)$ uniformly in $\tau$.     
\end{proposition}

To prove this proposition, we need to dissipate some $L^p$ norms at the discrete level of the JKO scheme. A useful tool for that purpose in the flow interchange lemma. We state a version of our interest which can be found in~\cite{dimarino}. 

\begin{lemma}[Flow interchange lemma]\label{lem:lp_mccann}
Let $\Omega\subset\mathbb{R}^d$ and let $\rho,\eta$ be absolutely continuous densities on $\Omega$.
Assume that $\rho\in Prox_{\widetilde{\mathcal{J}}}^{\tau}(\eta)$, and let $\varphi$ be a Kantorovich potential for the optimal transport from $\rho$ to $\eta$.
Then, for every $p>1$,
\begin{equation}\label{eq:FL1}
\int_{\Omega}\frac{\eta^p}{p-1}\,dx
 \ge 
\int_{\Omega}\frac{\rho^p}{p-1}\,dx
 +\tau 
p\int_{\Omega}\rho^{p-2}\,|\nabla\rho|^2\,dx
-\tau \frac{1}{p-1}\int_{\Omega}\rho^{p+1}\,dx,
\end{equation}
and for every $k>0$ and every $p\ge\frac{4d}{3d+1}$,
\begin{equation}\label{eq:FL2}
\int_{\Omega}\frac{(\eta-k)_{+}^{p}}{p-1}\,dx
 \ge \int_{\Omega}\frac{(\rho-k)_{+}^{p}}{p-1}\,dx
 +\tau 
p\int_{\{\rho\ge k\}}\rho^{p-2}\,|\nabla\rho|^2\,dx
-\tau \frac{1}{p-1}\int_{\{\rho\ge k\}}\rho^{p+1}\,dx.
\end{equation}
The inequalities are understood under sufficient regularity so that all terms make sense.
\end{lemma}

\begin{remark}
Note that in the context of $\widetilde{\mathcal{J}}$, compared to ~\cite{dimarino}, the penalization adds another term on the right-hand side of~\eqref{eq:FL1} and~\eqref{eq:FL2}. But this term is nonnegative and therefore can be neglected. 
\end{remark}

\begin{proof}[Proof of Proposition~\ref{lem:Linfty_JKO}]

Fix $T>0$ and let $n\ge 0$ be such that $n\tau\le T$. Set
\[
\eta:=\rho_\tau^{n},\qquad
\rho:=\rho_\tau^{n+1}.
\]

Note that to apply the flow interchange lemma, we need to a priori to prove the densities have enough regularity. This regularity is assumed for the initial condition. Then the flow interchange lemma and its proof (see for instance~\cite{Lisini-CH-gradient-flow}) provides directly the necessary regularity on each steps, and the lemma can be applied in our case. Applying~\eqref{eq:FL1} in Lemma~\ref{lem:lp_mccann} with exponent $p>1$ gives
\begin{equation}\label{eq:FI_discrete_simplified_new}
\int_\Omega \frac{\eta^p}{p-1}\,dx
\ge
\int_\Omega \frac{\rho^p}{p-1}\,dx
+\tau p\int_\Omega \rho^{p-2}|\nabla\rho|^2\,dx
-\tau \frac{1}{p-1}\int_\Omega \rho^{p+1}\,dx .
\end{equation}

\medskip
\underline{Step 1: Estimate of $\|\rho_\tau^n\|_{L^{d/2}}$ and coincidence of the two schemes.}
Let $d>2$ and fix a Sobolev constant $C_S>0$ such that, for all admissible $f$,
\begin{equation}\label{eq:sobolev_inhom}
\|f\|_{L^{\frac{2d}{d-2}}(\Omega)}^2
\le C_S\Bigl(\|\nabla f\|_{L^2(\Omega)}^2+\|f\|_{L^2(\Omega)}^2\Bigr).
\end{equation}
We estimate the source term as follows:
\begin{align*}
\int_\Omega \rho^{p+1}\,dx
&=\int_\Omega \rho\,(\rho^{p/2})^2\,dx\\
&\le \|\rho\|_{L^{d/2}(\Omega)}\ \|\rho^{p/2}\|_{L^{\frac{2d}{d-2}}(\Omega)}^2
\qquad\text{\Big(H\"older with exponents $\frac d2$ and $\frac{d}{d-2}$\Big)}\\
&\le C_S\,\|\rho\|_{L^{d/2}(\Omega)}
\Bigl(\|\nabla(\rho^{p/2})\|_{L^2(\Omega)}^2+\|\rho^{p/2}\|_{L^2(\Omega)}^2\Bigr)
\qquad\text{(by~\eqref{eq:sobolev_inhom})}\\
&= C_S\,\|\rho\|_{L^{d/2}(\Omega)}
\Bigl(\|\nabla(\rho^{p/2})\|_{L^2(\Omega)}^2+\|\rho\|_{L^p(\Omega)}^p\Bigr),
\end{align*}

Moreover,
\begin{equation}\label{eq:grad_identity_new}
p\int_\Omega \rho^{p-2}|\nabla\rho|^2\,dx
=\frac{4}{p}\int_\Omega |\nabla(\rho^{p/2})|^2\,dx.
\end{equation}
Multiplying~\eqref{eq:FI_discrete_simplified_new} by $(p-1)$ and using~\eqref{eq:grad_identity_new} and the bound on
$\int\rho^{p+1}$, we obtain
\begin{align*}
\int_\Omega \eta^p\,dx
&\ge \int_\Omega \rho^p\,dx
+\tau (p-1)\frac{4}{p}\|\nabla(\rho^{p/2})\|_{L^2(\Omega)}^2
-\tau\,C_S\|\rho\|_{L^{d/2}(\Omega)}
\Bigl(\|\nabla(\rho^{p/2})\|_{L^2(\Omega)}^2+\|\rho\|_{L^p(\Omega)}^p\Bigr)\\
&=\Bigl(1-\tau C_S\|\rho\|_{L^{d/2}(\Omega)}\Bigr)\int_\Omega \rho^p\,dx
+\tau\Bigl(\frac{4(p-1)}{p}-C_S\|\rho\|_{L^{d/2}(\Omega)}\Bigr)\|\nabla(\rho^{p/2})\|_{L^2(\Omega)}^2.
\end{align*}
Rearranging gives the key inequality
\begin{equation}\label{eq:FI_discrete2_new}
\Bigl(1-\tau C_S\|\rho\|_{L^{d/2}(\Omega)}\Bigr)\int_\Omega \rho^p\,dx
\le
\int_\Omega \eta^p\,dx
+\tau\Bigl(C_S\|\rho\|_{L^{d/2}(\Omega)}-\frac{4(p-1)}{p}\Bigr)\|\nabla(\rho^{p/2})\|_{L^2(\Omega)}^2.
\end{equation}
Let $p_0>d$. Using~\eqref{eq:JKO_Ld2_1}, \eqref{eq:JKO_Ld2_2}, together with the fact that the energy is nonincreasing along the JKO steps, i.e. $\widetilde{\mathcal{J}}[\rho_\tau^{n+1}]\le\widetilde{\mathcal{J}}[\rho_\tau^{n}]$ for all $n$, we have that for all $n$:
\begin{equation}\label{eq:Ld2_apriori_new}
\|\rho_\tau^{n}\|_{L^{d/2}(\Omega)}^2
\le
\frac{\nu}{C_0}\,\widetilde{\mathcal J}[\rho_0]+\frac{\nu}{C_0}M_d+\frac{1}{C_0}M_d
=:(C_{\nu,M_d})^2.
\end{equation}

Choose $p_0>d$ and assume that the parameters $\nu,M_d$ are such that
\begin{equation}\label{eq:smallness_choice_new}
C_S\,C_{\nu,M_d}\le \frac{4}{p_0}.
\end{equation}
Then for every $p\in(1,p_0]$ we have
\[
C_S\|\rho\|_{L^{d/2}}\le C_S C_{\nu,M_d}\le \frac{4(p-1)}{p},
\]
so that the second term on the right-hand side of~\eqref{eq:FI_discrete2_new} is nonpositive and can be dropped.

Assuming also $\tau$ is small enough such that $\tau(p_0-1)C_S C_{\nu,M_d}<1$, we conclude that for every $p\in(1,p_0]$,
\begin{equation*}
\int_\Omega \rho^p\,dx
\le \frac{1}{1-\tau C_S C_{\nu,M_d}}\int_\Omega \eta^p\,dx.
\end{equation*}
Iterating this inequality along the scheme gives, for $n\tau\le T$,
\begin{align*}
\|\rho_\tau^{n}\|_{L^p(\Omega)}^p
&\le \Bigl(1-\tau C_S C_{\nu,M_d}\Bigr)^{-n}\,\|\rho_0\|_{L^p(\Omega)}^p \\
&=\exp\!\Bigl(-n\log\bigl(1-\tau C_S C_{\nu,M_d}\bigr)\Bigr)\,\|\rho_0\|_{L^p}^p \notag\\
&\le \exp\!\Bigl(\frac{C_S C_{\nu,M_d}}{1-\tau C_S C_{\nu,M_d}}\,n\tau\Bigr)\,\|\rho_0\|_{L^p}^p \notag\\
&\le \exp\!\Bigl(\frac{C_S C_{\nu,M_d}}{1-\tau C_S C_{\nu,M_d}}\,T\Bigr)\,\|\rho_0\|_{L^p}^p. \notag
\end{align*}
In particular we get such a bound for $p=d/2$ and $p=p_0$. To have that the original scheme and the penalize scheme coïncide, that is $\widetilde{\mathcal J}[\rho_\tau^{n}] = \mathcal{J}[\rho_\tau^n]$ we need to enforce $\|\rho_\tau^{n}\|_{L^{d/2}}^2\le \frac{M_d}{C_0}$ for all $n$. 
This requires that $\|\rho_0\|_{L^{d/2}}$ is small enough, depending on the final time $T$. This ends the proof of the first item of the proposition.

\medskip

\underline{Step 2: Discrete Alikakos iteration and the $L^\infty$ bound.}

Fix $p\ge 2$ (in particular $p\ge \frac{4d}{3d+1}$) and $k>0$ and set
\[
A_n(p,k):=\int_\Omega (\rho_\tau^{n}-k)_+^p\,dx.
\]
Applying~\eqref{eq:FL2} with $(\eta,\rho)=(\rho_\tau^n,\rho_\tau^{n+1})$ and multiplying by $(p-1)$ we obtain
\begin{equation}\label{eq:FL2_multiplied_new}
A_{n+1}(p,k)
\le A_n(p,k)
-\tau\,p(p-1)\int_{\{\rho\ge k\}}\rho^{p-2}|\nabla\rho|^2\,dx
+\tau\int_{\{\rho\ge k\}}\rho^{p+1}\,dx.
\end{equation}
Also we note
\[
p(p-1)\int_{\{\rho\ge k\}}\rho^{p-2}|\nabla\rho|^2\,dx
=\frac{4(p-1)}{p}\int_{\{\rho\ge k\}}|\nabla(\rho^{p/2})|^2\,dx.
\]

Write $\rho^{p+1}=\rho\cdot(\rho^{p/2})\cdot(\rho^{p/2})$ and apply H\"older on $\{\rho\ge k\}$ with exponents
$d$, $\frac{2d}{d-2}$ and $2$ (since $\frac1d+\frac{d-2}{2d}+\frac12=1$):
\[
\int_{\{\rho\ge k\}}\rho^{p+1}\,dx
\le
\|\rho\|_{L^d(\rho\ge k)}\,
\|\rho^{p/2}\|_{L^{\frac{2d}{d-2}}(\rho\ge k)}\,
\|\rho^{p/2}\|_{L^{2}(\rho\ge k)}.
\]
Set $f:=\rho^{p/2}$. By Sobolev~\eqref{eq:sobolev_inhom} (restricted to $\{\rho\ge k\}$ but the Sobolev constant is the same at the critical exponent),
\[
\|f\|_{L^{\frac{2d}{d-2}}(\rho\ge k)}
\le C\Bigl(\|\nabla f\|_{L^2(\rho\ge k)}+\|f\|_{L^2(\rho\ge k)}\Bigr).
\]
Therefore
\begin{align}
\int_{\{\rho\ge k\}}\rho^{p+1}\,dx
&\le C\,\|\rho\|_{L^d(\rho\ge k)}
\Bigl(\|\nabla(\rho^{p/2})\|_{L^2(\rho\ge k)}+\|\rho^{p/2}\|_{L^2(\rho\ge k)}\Bigr)\,
\|\rho^{p/2}\|_{L^{2}(\rho\ge k)} \notag\\
&= C\,\|\rho\|_{L^d(\rho\ge k)}\,\|\nabla(\rho^{p/2})\|_{L^2(\rho\ge k)}\,\|\rho\|_{L^p(\rho\ge k)}^{p/2}
+ C\,\|\rho\|_{L^d(\rho\ge k)}\,\|\rho\|_{L^p(\rho\ge k)}^{p}.
\label{eq:source_term_split_new}
\end{align}
Apply Young's inequality to the first term in~\eqref{eq:source_term_split_new}  yields
\begin{equation}\label{eq:source_after_young_new}
\int_{\{\rho\ge k\}}\rho^{p+1}\,dx
\le
\frac{4(p-1)}{p}\|\nabla(\rho^{p/2})\|_{L^2(\rho\ge k)}^2
+ C\,\|\rho\|_{L^p(\rho\ge k)}^{p}\,\|\rho\|_{L^d(\rho\ge k)}^{2}.
\end{equation}
Finally, on $\{\rho\ge k\}$ we have $\rho\le (\rho-k)_+ + k$, hence by $(a+b)^p\le 2^{p-1}(a^p+b^p)$,
\[
\int_{\{\rho\ge k\}}\rho^p\,dx
\le C^p\int_\Omega (\rho-k)_+^p\,dx + Ck^p = C^p A_{n+1}(p,k)+Ck^p,
\]
and therefore
\begin{equation}\label{eq:Lp_tail_by_trunc_new}
\|\rho\|_{L^p(\rho\ge k)}^{p}\le C^p A_{n+1}(p,k)+C^pk^p.
\end{equation}

Combining~\eqref{eq:source_after_young_new}--\eqref{eq:Lp_tail_by_trunc_new} and inserting into~\eqref{eq:FL2_multiplied_new},
the gradient terms cancel and we obtain
\begin{equation}\label{eq:trunc_recursion_new}
A_{n+1}(p,k)
\le A_n(p,k)
+ C^p\tau (p-1)\bigl(A_{n+1}(p,k)+k^p\bigr)\,\|\rho\|_{L^d(\rho\ge k)}^{2}.
\end{equation}

We can now work as in~\cite{dimarino}. Let $p_0>d$ as in Step~2, and denote
\[
M_{p_0}:=\sup_{m\tau\le T}\|\rho_\tau^{m}\|_{L^{p_0}(\Omega)} <\infty.
\]
For any $k>0$, since $p_0>d$ we have on $\{\rho\ge k\}$ the pointwise bound $\rho^d \le k^{d-p_0}\rho^{p_0}$, hence
\[
\|\rho\|_{L^d(\rho\ge k)}^d=\int_{\{\rho\ge k\}}\rho^d\,dx
\le k^{d-p_0}\int_\Omega \rho^{p_0}\,dx
\le k^{d-p_0} M_{p_0}^{p_0}.
\]
Therefore
\begin{equation*}
\|\rho\|_{L^d(\rho\ge k)}^{2}
\le
M_{p_0}^{2p_0/d}\,k^{2(1-p_0/d)}.
\end{equation*}
Choose $k=k(p)$ so that
\begin{equation}\label{eq:k_choice_new}
C^p(p-1)\,\|\rho\|_{L^d(\rho\ge k(p))}^{2}\le \frac12,
\end{equation}
that is
\begin{equation}\label{eq:k_of_p_new}
k(p):=
\Bigl(2C^p(p-1)\Bigr)^{\frac{d}{2(p_0-d)}}\,M_{p_0}^{\frac{p_0}{p_0-d}}.
\end{equation}
Then~\eqref{eq:trunc_recursion_new} and~\eqref{eq:k_choice_new} imply
\begin{equation*}
\Bigl(1-\frac{\tau}{2}\Bigr)A_{n+1}(p,k(p))
\le A_n(p,k(p)) + \frac{\tau}{2}k(p)^p.
\end{equation*}
Iterating this formula for $n\tau\le T$ yields
\begin{equation}\label{eq:trunc_iterated_new}
\sup_{n\tau\le T}A_n(p,k(p))
\le C\bigl(A_0(p,k(p))+k(p)^p\bigr)
\le C\,k(p)^p,
\end{equation}
where in the last inequality we absorbed $A_0(p,k(p))$ into $k(p)^p$ by increasing the constant in~\eqref{eq:k_of_p_new} if necessary since $A_0$ depends only on the initial condition.

Finally, we obtain
\[
\sup_{n\tau\le T}\|\rho_\tau^{n}\|_{L^p(\Omega)}^p
\le C^p\sup_{n\tau\le T}A_n(p,k(p)) + C^pk(p)^p
\le C\,k(p)^p,
\]
up to enlarging $k$.

Set $p_0>d$ as above and define $p_j:=2^j p_0$ for $j\ge 0$, and
\[
D_j:=\sup_{n\tau\le T}\|\rho_\tau^{n}\|_{L^{p_j}(\Omega)}.
\]
Applying the previous estimate with $p=p_{j+1}$ and using the tail bound derived from $D_j$ (instead of $M_{p_0}$),
one obtains thresholds $k_{j+1}$ such that
$$ k_{j+1}\le (2^{j}p_0 C)^{d/(2(2^{j}p_0-d))}k_j^{1+\frac{d}{2^{j}p_0-d}},\quad D_{j+1}\le C^{1/p_{j+1}}k_{j+1}. 
$$

Thus $$ \log_{+}(k_{j+1})\le C\frac{jd}{2^{j+1}p_0-2d} + \left(1+\frac{d}{2^{j}p_0-d}\right)\log_{+}k_j. $$ Since $\sum_j \frac{1}{2^{j}}<+\infty$ and $\sum_j \frac{j}{2^{j}}<+\infty$ we obtain that $\lim_{j\to+\infty}k_j<+\infty$. This concludes the proof.

\end{proof}

}

To prove the $W^{1,p}(\Omega)$ bounds we state some general lemmas. 

\begin{lemma}\label{lem:link_kantorovich}
Let $\rho,\eta>0$ with $\rho\in Prox_{\widetilde{\mathcal{J}}}^{\tau}[\eta]$ a step of the JKO scheme. And let $\varphi$ the Kantorovich potential between $\rho$ and $\eta$. Then
$$
\f{1}{\tau^{p}}W_{p}(\rho,\eta)^{p} = \int_{\Omega} \rho \left|\f{\nabla\varphi}{\tau}\right|^{p} \diff x = \mathcal{F}_{p}[\rho], 
$$
where $\mathcal{F}_{p}$ is defined in~\eqref{eq:FP}.
\end{lemma}

Its proof is just a consequence of the definition of the Wasserstein distance{\color{blue}, using $x-T(x)=\nabla\varphi$} and the formula $\f{\varphi}{\tau}+ \log \rho +1 -u[\rho] = C$ where $C$ is a constant. This equality is satisfied on $\{\rho>0\}$. The previous lemma shows that we can control the $p$-Wasserstein distance between two steps of the JKO scheme by the functional $F_{p}$. The Wasserstein distances can be related to the norms of duals of Sobolev spaces. 

\begin{lemma}\label{lem:sobone} Let $\rho, \eta\in L^{\infty}(\Omega)$ be two absolutely continuous densities. Then for $\f{1}{p}+\f{1}{q}=1$,
$$\| \rho- \eta\|_{(W^{1,p})^*(\Omega)} \leq C( \| \rho\|_{\infty} , \| \eta \|_{\infty}) W_q(\rho,\eta).$$
\end{lemma}

Actually, the previous estimate is far from being sharp, and only some suitable $L^{r}$ norms of $\rho,\eta$ are required on the right-hand side. We refer to~\cite[Proposition 5.9]{dimarino}.

To reproduce the computation of the dissipation of the $\mathcal{F}_p$ in Proposition~\ref{prop:entropy_equality} at the discrete level, we present the so-called five-gradients inequality that was introduced in~\cite{de2016bv} to obtain $BV$ estimates for some variational problems arising in optimal transport. 

\begin{lemma}\label{lem:five_gradients_inequality}
Let $\Omega$ be a smooth open bounded convex domain, and $\rho$, $\eta\in W^{1,1}(\Omega)$, be two strictly positive densities. Let $H\in C^{1}(\R^d)$ be a radially symmetric convex function. Then 
 \begin{multline*}
 \int_{\Omega}( \nabla \rho \cdot \nabla H(\nabla \varphi)+\nabla \eta\cdot \nabla H(\nabla \psi) \diff x\\= \int_{\Omega}\rho \mathrm{Tr}[D^2H(\nabla\varphi)\cdot (D^2\varphi)^2 \cdot(I-D^2\varphi)^{-1}]\diff x+ \int_{\partial\Omega}\left(\rho\nabla H(\nabla\varphi) + g \nabla H(\nabla \psi)\right)\cdot \vec{n} \diff \mathcal{H}^{d-1}\geq 0
 \end{multline*}
 where $(\varphi,\psi)$ are the corresponding Kantorovich potentials in the transport from $\rho$ to $\eta$.
\end{lemma}

Finally, as in~\cite[Lemma 5.1]{dimarino} and~\cite[Proposition 5.2]{toshpulatov} we have the following lemma.

\begin{lemma}\label{lem:consequence_five_gradients}
Let $\rho^{n+1}_{\tau}\in Prox_{\widetilde{\mathcal{J}}}^{\tau}(\rho^{n}_{\tau})$ be {\color{blue} any defined sequence of the JKO scheme, with $\tau$ small enough and with initial value $\rho_{0}$ as in Theorem~\ref{lem:Linfty_JKO} Let $\varphi,\psi$ be the Kantorovich potentials from the transport from $\rho^{n+1}_{\tau}$ to $\rho^{n}_{\tau}$ and assume.  Let $H:\R^d\to \R$ be a radial convex function. Set
$$
Z_{\rho^{n+1}_{\tau}} := \nabla (\log \rho^{n+1}_{\tau} -u[\rho^{n+1}_{\tau}]), \quad Z_{\rho^{n}_{\tau}} := \nabla(\log \rho^{n}_{\tau} - u[\rho^{n}_{\tau}]). 
$$

Then 
$$
\int_{\Omega}H(Z_{\rho^{n}_{\tau}})\diff\rho^{n}_{\tau} \ge \int_{\Omega}H(Z_{\rho^{n+1}_{\tau}})\diff\rho^{n+1}_{\tau} + \int_{\Omega}\nabla H\left(\f{\nabla\varphi}{\tau}\right)\cdot \left(-\nabla u[\rho^{n+1}_{\tau}] + \nabla u[\rho^{n}_{\tau}]\circ T\right)\diff x + R, 
$$

where $R\ge 0$ is the remainder term, 

\begin{align*}
R &:= \f{1}{\tau}\int_{\Omega}\rho^{n+1}_{\tau} \mathrm{Tr}[D^2H\left(\f{\nabla\varphi}{\tau}\right)\cdot (D^2\varphi)^2 \cdot(I-D^2\varphi)^{-1}]\diff x\\
&+\int_{\p\Omega}\rho^{n+1}_{\tau}\nabla H \left(\f{\nabla\varphi}{\tau}\right)\cdot \vec{n} +  \rho^{n}_{\tau} \nabla H \left(\f{\nabla\psi}{\tau}\right)\cdot\vec{n}\diff  \mathcal{H}^{d-1}. 
\end{align*}}
\end{lemma}

With these results, we are now in position to prove Theorem~\ref{thm:JKO}. 

\subsection{$W^{1,p}$ estimates and maximum principle}

We start by proving the $L^{\infty}(0,T; W^{1,p}(\Omega))$ estimate on the curve $\rho_{\tau}$ defined by the JKO scheme with the five-gradients inequality. For now, we just need that in Lemma~\ref{lem:five_gradients_inequality}, the left-hand side of the inequality (which is the remainder term) is nonnegative. This uniform bound allows us to prove that the sequence constructed in the JKO scheme is bounded from below by a positive constant. This in turn allows us to prove the strong convergence of the JKO scheme in $L^2(0,T; H^2(\Omega))$. In this subsection we focus on the $L^{\infty}(0,T; W^{1,p}(\Omega))$ and the bound from below. 

\begin{proposition}\label{prop:JKO_FP}
Let $\rho^{n+1}_{\tau}\in Prox_{\widetilde{\mathcal{J}}}^{\tau}(\rho^{n}_{\tau})$ be {\color{blue} any defined sequence of the JKO scheme, with $\tau$ small enough and with initial value $\rho_{0}$ as in Theorem~\ref{lem:Linfty_JKO}}. Then for all $1\le p<+\infty$ there exists $C$ such that for all $n$: 
$$
\mathcal{F}_{p}[\rho^{n}_{\tau}]\le C .
$$

In particular $\rho^{n}_{\tau}$ is bounded in $W^{1,p}(\Omega)$ uniformly in $n$ and the interpolation curve $\rho_{\tau}$ is bounded in $L^{\infty}(0,T; W^{1,p}(\Omega))$ uniformly in $\tau$. 

\end{proposition}

\begin{remark}
The assumption that $\tau$ is sufficiently small can be explicitly computed, as it will become clear in the proof. Since our goal is to take the limit as $\tau \to 0$, this assumption is always valid. 
\end{remark}

\begin{proof}
{\color{blue}Here we use the $L^{\infty}$ bounds on $\rho_\tau^n$ from Proposition~\ref{lem:Linfty_JKO}.} We prove the proposition for $p>d$ but an estimate on $F_p$ implies an estimate on $F_q$ for $q<p$. Also, the estimate on $\mathcal{F}_{p}[\rho^{n}_{\tau}]$ implies an estimate in $W^{1,p}(\Omega)$ as a consequence of Lemma~\ref{lem:rho_W1p}. Choosing $H(z)=|z|^p$ in Lemma~\ref{lem:consequence_five_gradients} {\color{blue} and observing the remainder term is nonnegative} yields 
$$
\mathcal{F}_{p}[\rho^{n}_{\tau}] \ge \mathcal{F}_{p}[\rho^{n+1}_{\tau}] + \f{p}{\tau^{p-1}}\int_{\Omega} \rho^{n+1}_{\tau} |\nabla\varphi|^{p-2}\nabla\varphi\cdot (-\nabla u[\rho^{n+1}_{\tau}] + \nabla u[\rho^{n}_{\tau}] \circ T)\diff x.
$$

We decompose the last term of the right-hand side into two terms:

$$
\f{p}{\tau^{p-1}}\int_{\Omega} \rho^{n+1}_{\tau} |\nabla\varphi|^{p-2}\nabla\varphi\cdot (-\nabla u[\rho^{n+1}_{\tau}] + \nabla u[\rho^{n}_{\tau}] \circ T)\diff x = C_1 + C_2
$$

where 
\begin{align*}
&C_1 := -\f{p}{\tau^{p-1}}\int_{\Omega} \rho^{n+1}_{\tau} |\nabla\varphi|^{p-2}\nabla\varphi\cdot (\nabla u[\rho^{n}_{\tau}]- \nabla u[\rho^{n}_{\tau}] \circ T)\diff x, \\
&C_2 := -\f{p}{\tau^{p-1}}\int_{\Omega} \rho^{n+1}_{\tau} |\nabla\varphi|^{p-2}\nabla\varphi\cdot (\nabla u[\rho^{n+1}_{\tau}-\rho^{n}_{\tau}])\diff x.
\end{align*}

We start by estimating $C_1$. {\color{blue}Using Taylor-Young formula and the definition $T(x) = x-\nabla\varphi(x)$ we obtain $\nabla u[\rho^{n}_{\tau}](x) - \nabla u[\rho^{n}_{\tau}] \circ T(x) =  \int_{0}^{1} D^2 u(tx + (1-t)T(x))\nabla\varphi(x)\diff t$.

Therefore using Lemma~\ref{lem:link_kantorovich} we can estimate 

$$
|C_1|\le \tau p\|D^{2}u[\rho^{n}_{\tau}]\|_{L^{\infty}}\int_{\Omega} \rho^{n+1}_{\tau} \left|\f{\nabla\varphi}{\tau}\right|^{p}\diff x =\tau p \|D^{2}u[\rho^{n}_{\tau}]\|_{L^{\infty}}\mathcal{F}_{p}[\rho^{n+1}_{\tau}].
$$
}

As in the continuous setting, using Theorem~\ref{thm:functional_inequality}, Lemma~\ref{lem:rho_W1p} and Theorem~\ref{lem:Linfty_JKO}, we obtain
$$
|C_1| \le C\tau\mathcal{F}_{p}[\rho^{n+1}_{\tau}](1+\log(\mathcal{F}_{p}[\rho^{n}_{\tau}]+1)).
$$

We turn our attention to $C_2$. We can estimate $C_2$ from above as
$$
|C_2| \le \tau p\int_{\Omega} \rho^{n+1}_{\tau} \left|\f{\nabla\varphi}{\tau}\right|^{p-1} \left|\nabla u\left[\f{\rho^{n+1}_{\tau}-\rho^{n}_{\tau}}{\tau}\right]\right|\diff x.
$$

By Hölder inequality with exponents $\f{p}{p-1}$ and $p$ and using Lemma~\ref{lem:link_kantorovich}
\begin{align*}
|C_2| &\le C \tau p \mathcal{F}_{p}[\rho^{n+1}_{\tau}]^{1-1/p}\left\|\nabla u\left[\f{\rho^{n+1}_{\tau}-\rho^{n}_{\tau}}{\tau}\right]\right\|_{L^{p}}. 
\end{align*}

Then combining Lemma~\ref{lem:link_kantorovich} with $\f{1}{p} + \f{1}{q}=1$, Lemma~\ref{lem:sobone}, and Lemma~\ref{lem:sobone2} 

$$
\left\|\nabla u\left[\f{\rho^{n+1}_{\tau}-\rho^{n}_{\tau}}{\tau}\right]\right\|_{L^{p}} \le C\left\|\f{\rho^{n+1}_{\tau}-\rho^{n}_{\tau}}{\tau}\right\|_{(W^{1,q})^*}\le \f{C}{\tau} W_{p}(\rho^{n+1}_{\tau},\rho^{n}_{\tau}) = CF_{p}[\rho^{n+1}_{\tau}]^{1/p}.
$$

In the end
$$
|C_2|\le C\tau \mathcal{F}_p[\rho^{n+1}_{\tau}].
$$

Therefore we deduce a discrete version of the differential inequality $\f{d\mathcal{F}_{p}[\rho]}{dt}\le C(\mathcal{F}_{p}[\rho]+1)\log(\mathcal{F}_{p}[\rho]+1)$ that is
$$
\mathcal{F}_p[\rho^{n}_{\tau}]\ge \mathcal{F}_p[\rho^{n+1}_{\tau}] -C\tau \mathcal{F}_p[\rho^{n+1}_{\tau}]\log(\mathcal{F}_p[\rho^{n}_{\tau}]+1) - C\tau \mathcal{F}_p[\rho^{n+1}_{\tau}]
$$

This define the sequence $u_n = \mathcal{F}_{p}[\rho_n]$ by
$$
u_{n+1}\le \f{u_n}{1-C\tau-C\tau \log(u_n+1)}.
$$
Now we define $v_n = \log(u_n +1)$ and we obtain 
$$
u_{n+1}+1\le \f{u_n+1}{1-C\tau-C\tau v_n} -\f{C\tau + C\tau v_n}{1-C\tau-C\tau v_n}. 
$$

We can choose $\tau$ small enough to make the last term negative (taking $\tau< \f{1}{C(1+v_n)}$). At first $\tau$ seems to depend on $n$ but we will see that this implies a uniform estimate on all $v_n$ independent of $\tau$ and then we can conclude that $\tau$ can be chosen in fact independently of $n$. Taking the $\log$ of the previous expression then yields

$$
v_{n+1}\le v_n - \log(1-C\tau-C\tau v_n)\le v_n + C\tau + C\tau v_n.
$$

Therefore
$$
v_{n+1} \le  v_n(1+C\tau)+ C\tau.
$$

By induction 
$$
v_{n}\le v_{0}(1+C\tau)^{n} + (1+C\tau)^n-1\le v_{0} e^{nC\tau}+e^{nC\tau}-1,
$$

And we conclude to a uniform estimate on $u_n$ since $n\tau\le T$. 

\end{proof}

We use this proposition to prove a maximum principle:

\begin{proposition}\label{prop:JKO_maximum_principle}
Let $\rho_{0}$ be as in Theorem~\ref{thm:JKO}. Let $\rho^{n+1}_{\tau}\in Prox_{\widetilde{\mathcal{J}}}^{\tau}(\rho^{n}_{\tau})$ be the JKO step and $\varphi^{n+1}_{\tau}$ be the Kantorovich potential from $\rho^{n+1}_{\tau}$ to $\rho^{n}_{\tau}$
$$
\log\rho^{n+1}_{\tau} - u[\rho^{n+1}_{\tau}] + \f{\varphi^{n+1}_{\tau}}{\tau}=C,
$$
{\color{blue} where the inequality follows since $\rho_{\tau}^{n+1}$ is positive from Lemma~\ref{lem:strict_positivy}}.
Then there exists $0<c_1\le c_2$ independent of $\tau$ such that for all $n$
$$
c_1\le \rho^{n}_{\tau}\le c_2.
$$
\end{proposition}

\begin{proof}
The upper bound follows from assumption. Now, we focus our attention on the lower bound. {\color{blue}We first have positivity from Lemma~\ref{lem:strict_positivy}}. Now we consider a minimal point $x$ of  $\log(\rho^{n+1}_{\tau}) -  u[\rho^{n+1}_{\tau}]$, which corresponds to a maximum for $\varphi^{n+1}_{\tau}$. This minimum cannot occur at the boundary, the argument is classical: since $\nabla \varphi^{n+1}_{\tau}$ points outward the transport map $T(x) = x - \nabla \varphi^{n+1}_{\tau}(x)$ would send $x$ to the opposite side of the boundary (we assumed $\Omega$ convex), contradicting the monotonicity of $T$. For further details, see~\cite[Proof of Lemma 2.4]{iacobelli2019weighted}. Therefore the minimal point $x$ in in the interior of $\Omega$. Applying the Monge-Ampère equation, 
$$
\rho^{n+1}_{\tau}(x) = \rho^{n}_{\tau}(x-\nabla\varphi^{n+1}_{\tau}(x))\det(I-D^{2}\varphi^{n+1}_{\tau}(x)),
$$

together with the fact that $\nabla\varphi^{n+1}_{\tau}(x) =0$ and $D^2 \varphi^{n+1}_{\tau}(x) \le 0$, we derive 
$$
\log\rho^{n+1}_{\tau}(x)\ge \log \rho^{n}_{\tau}(x).
$$ Then we add and subtract $u[\rho^{n+1}_{\tau}]$ and take the min on the right hand side
$$
\min [\log(\rho^{n+1}_{\tau})- u[\rho^{n+1}_{\tau}]] \ge \min [\log(\rho^{n}_{\tau}) - u[\rho^{n}_{\tau}]] - \|u[\rho^{n+1}_{\tau}]-u[\rho^{n}_{\tau}]\|_{L^{\infty}}.
$$

Therefore we need to control
$$
\|u[\rho^{n+1}_{\tau}-\rho^{n}_{\tau}]\|_{L^{\infty}}.
$$
For $p>d$, by Sobolev embedding and Lemma~\ref{lem:sobone}, Lemma~\ref{lem:sobone2}.
\begin{align*}
\|u[\rho^{n+1}_{\tau}-\rho^{n}_{\tau}]\|_{L^{\infty}}&\le \|u[\rho^{n+1}_{\tau}-\rho^{n}_{\tau}]\|_{W^{1,p}} \\
&\le C\mathcal{W}_{p}(\rho^{n+1}_{\tau},\rho^{n}_{\tau}).
\end{align*}

Now with Lemma~\ref{lem:link_kantorovich} and Proposition~\ref{prop:JKO_FP} we deduce
$$
\|u[\rho^{n+1}_{\tau}-\rho^{n}_{\tau}]\|_{L^{\infty}}\le C\tau. 
$$
We conclude that  for all $n$
$$
\min [\log\rho^{n+1}_{\tau} -u[\rho^{n+1}_{\tau}]] \ge \min[\log(\rho^{n}_{\tau}) - u[\rho^{n}_{\tau}]] + C\tau
$$
Since $n\tau\le T$, that means that there exists a constant $C$ such that for all $n$:
$$
\min [\log\rho^{n+1}_{\tau} - u[\rho^{n+1}_{\tau}]] \ge \min [\log\rho^{0}-u[\rho^{0}]]+C 
$$

And by a simple computation using the $L^{\infty}$ estimates of $u[\rho_{n+1}^{\tau}]$ we proved that $\log\rho^{n+1}_{\tau}$ is bounded uniformly from below by a constant (which could be negative). But taking the exponential yields the result. 

\end{proof}

To finish the proof of Theorem~\ref{thm:JKO}, it only remains to prove the strong $L^{2}(0,T; H^{2}(\Omega))$ convergence of the scheme.

\subsection{Strong convergence of the JKO scheme in $L^{2}(0,T; H^{2}(\Omega))$}

The following proposition is the integrated discrete version of the entropy equality stated in Proposition~\ref{prop:entropy_equality} in the case $p=2$.

\begin{proposition}\label{prop:L^2H^2bound}
There exists a function $\eps(\tau)$ which tends to 0 when $\tau\to 0$ such that the constructed curve $\rho_{\tau}$ satisfies
\begin{align*}
 &\mathcal{F}_{2}[\rho_{0}] -  \mathcal{F}_{2}[\rho_{\tau}(T)]\ge 2\int_{0}^{T}\int_{\Omega} |D^{2}(\log \rho_{\tau}-u[\rho_\tau])|^2\rho_{\tau}\diff x \diff t \\
 &- 2 \int_{0}^{T}\int_{\Omega}\rho_{\tau}  \nabla(\log(\rho_{\tau}) - u[\rho_{\tau}])\cdot D^{2}u[\rho_{\tau}]\nabla(\log(\rho_{\tau})-u[\rho_{\tau}])\diff x\diff t\\
 & + 2 \int_{0}^{T}\int_{\Omega}\rho_{\tau}^{t}(x)\nabla(\log(\rho_{\tau})-u[\rho_{\tau}])\cdot\nabla u\left[ \f{\rho_{\tau}-\rho_{t-\tau}^{\tau}}{\tau}\right]\diff x \diff t \\
 &+ 2\int_{0}^{T}\int_{\p\Omega}\rho \f{\nabla (\log \rho_{\tau}- u[\rho_{\tau}])\cdot D^{2} h \nabla (\log \rho_{\tau}- u[\rho_{\tau}])}{|\nabla h|}\diff \mathcal{H}^{d-1}\diff t + \eps(\tau)
\end{align*}
\end{proposition}

To prove this proposition, we need to apply once again the five gradients inequality, and show how the discrete terms relate to their continuous counterpart. For this, we need a control on $|x-T(x)| = |\nabla\varphi(x)|$. The following lemma can be adapted from~\cite[Proposition 5.3]{toshpulatov} and uses strongly the bound from below found in Proposition~\ref{prop:JKO_maximum_principle}. Indeed, the result from~\cite{Bouchitt2007ANL} implies that $\|x-T(x)\|_{L^{\infty}}$ can be controlled by the Wasserstein distance $W_{2}(\rho_{\tau}^{n}, \rho_{\tau}^{n+1})$ provided the densities are bounded from below. 
\begin{lemma}\label{lem:control_phi}
 Let $\rho^{n+1}_{\tau}\in Prox_{\widetilde{\mathcal{J}}}[\rho^{n}_{\tau}]$ be any steps constructed in the JKO scheme with initial condition $\rho_0$ as in Theorem~\ref{thm:JKO}. Then there exists $\beta>0$  and $C>0$ such that all the potentials $\varphi_{\tau}^{n}$ satisfy
$$\|\nabla\varphi_{\tau}^{n}\|_{L^{\infty}} + \|D^2\varphi_{\tau}^{n}\|_{L^{\infty}}\le C\tau^{\beta}$$.
\end{lemma}

The estimate for $\nabla\varphi_{\tau}^{n}$ can be derived from~\cite{Bouchitt2007ANL}, together with $W_{2}^{2}(\rho_{\tau}^{n},\rho_{\tau}^{n+1})\leq C\tau$, which comes from the optimality condition in the JKO scheme. Since $\rho_{\tau}^{n+1}$ and $\rho_{\tau}^{n}$ are uniformly controlled in $W^{1,p}$ and thus in $C^{0,\alpha}$, the Monge-Ampère equation yields a $C^{2,\alpha}$ estimate for $\varphi_{\tau}^{n}$. The estimate on $D^{2}\varphi_{\tau}^{n}$ is then achieved by interpolation with this result and the earlier estimate for $\nabla\varphi_{\tau}^{n}$.

\begin{proof}[Proof of Proposition~\ref{prop:L^2H^2bound}]
Let $\rho^{n+1}_{\tau}\in Prox_{J}[\rho^{n}_{\tau}]$ and $(\varphi^{n+1}_{\tau},\psi^{n+1}_{\tau}$ the corresponding Kantorovich potentials). As in the proof of Proposition~\ref{prop:JKO_FP}, we use the five gradient inequality but with $H(z)=|z|^2$ and we consider the remainder term. 
We obtain similarly
$$
\mathcal{F}_{2}[\rho_{\tau}^{n}]\ge \mathcal{F}_{2}[\rho_{\tau}^{n+1}] + C_1 + C_2 + R
$$

where 

\begin{align*}
&C_1 := -\f{2}{\tau}\int_{\Omega} \rho_{\tau}^{n+1} \nabla\varphi_{\tau}^{n+1}\cdot (\nabla u[\rho_{\tau}^{n}]  -\nabla u[\rho_{\tau}^{n}] \circ T)\diff x, \\
&C_2 :=- \f{2}{\tau}\int_{\Omega} \rho_{\tau}^{n+1} \nabla\varphi_{\tau}^{n+1}\cdot \nabla u[\rho_{\tau}^{n+1}-\rho_{\tau}^{n}]\diff x, \\
&R := 2\f{1}{\tau}\int_{\Omega}\rho_{\tau}^{n+1} \mathrm{Tr}[(D^2\varphi_{\tau}^{n+1})^2 \cdot(I-D^2\varphi_{\tau}^{n+1})^{-1}]\diff x\\
&+2\int_{\p\Omega}\rho_{\tau}^{n+1} \f{\nabla\varphi_{\tau}^{n+1}}{\tau}\cdot \vec{n} +  \rho_{\tau}^{n} \f{\nabla\psi_{\tau}^{n+1}}{\tau}\cdot\vec{n}\diff \mathcal{H}^{d-1}. 
\end{align*}

Concerning $C_1$ we write
\begin{align*}
&C_1 =  -\f{2}{\tau}\int_{\Omega} \rho_{\tau}^{n+1} \nabla\varphi_{\tau}^{n+1}\cdot (\nabla u[\rho_{\tau}^{n+1}]- \nabla u[\rho_{\tau}^{n+1}] \circ T) \diff x\\
&- \f{2}{\tau}\int_{\Omega} \rho_{\tau}^{n+1} \nabla\varphi_{\tau}^{n+1}\cdot (\nabla u[\rho_{\tau}^{n}-\rho_{\tau}^{n+1}]- \nabla u[\rho_{\tau}^{n}-\rho_{\tau}^{n+1}] \circ T)\diff x \\
&= C_{11} + C_{12}. 
\end{align*}

For $C_{11}$, with $T(x)= x-\nabla\varphi^{n+1}_{\tau}(x)$ and by applying twice Taylor-Young formula we can write
\begin{align*}
\nabla u[\rho^{n+1}_{\tau}] - \nabla u[\rho^{n+1}_{\tau}]\circ T &= \int_{0}^{1} D^{2}u[\rho^{n+1}_{\tau}](x- t\nabla\varphi^{n+1}_{\tau})\nabla\varphi^{n+1}_{\tau}\diff t\\
&= D^{2}u[\rho^{n+1}_{\tau}](x)\nabla\varphi^{n+1}_{\tau}(x) \\ &- \int_{0}^{1}t\int_{0}^{1} \nabla D^{2} u[\rho^{n+1}_{\tau}](x-st\nabla\varphi^{n+1}_{\tau}): \nabla\varphi^{n+1}_{\tau}\otimes\nabla\varphi^{n+1}_{\tau}\diff s \diff t.
\end{align*}

Therefore 

\begin{align*}
C_{11}& = -\f{2}{\tau}\int_{\Omega}\rho_{\tau}^{n+1}\nabla\varphi_{\tau}^{n+1}\cdot D^{2}u[\rho^{n+1}_{\tau}]\nabla\varphi^{n+1}_{\tau}\diff x + I\\
&= -2\tau\int_{\Omega}\rho_{\tau}^{n+1}\nabla(\log(\rho_{\tau}^{n+1}) - u[\rho_{\tau}^{n+1}])\cdot D^{2}u[\rho^{n+1}_{\tau}]\nabla(\log(\rho_{\tau}^{n+1}) - u[\rho_{\tau}^{n+1}])\diff x + I. 
\end{align*}
where 
$$
|I|\le C\f{2}{\tau}\int_{\Omega}\rho_{\tau}^{n+1}|\nabla\varphi_{\tau}^{n+1}|^3\int_{0}^{1}\int_{0}^{1}|\nabla D^{2} u[\rho_{\tau}^{n+1}](x-st\nabla\varphi_{\tau}^{n+1})|\diff s\diff t \diff x.
$$

By Hölder inequality and using the $L^{\infty}$ norm on $\rho_{\tau}^{n+1}$ with Theorem~\ref{lem:Linfty_JKO}

$$
|I|\le C\f{2}{\tau}\left(\int_{\Omega}\rho_{\tau}^{n+1}|\nabla\varphi_{\tau}^{n+1}|^6\right)^{1/2}\left(\int_{\Omega}\int_{0}^{1}\int_{0}^{1}|\nabla D^{2} u[\rho_{\tau}^{n+1}](x-st\nabla\varphi_{\tau}^{n+1})|^2\diff s\diff t \diff x.\right)^{1/2}
$$

With Lemma~\ref{lem:link_kantorovich} and a change of variables: 
$$
|I|\le C\tau^{2}\mathcal{F}_{6}[\rho_{\tau}^{n+1}]^{1/2}\|\nabla D^{2}u[\rho_{\tau}^{n+1}]\|_{L^{2}}(1+\|D^{2}\varphi_{\tau}^{n+1}\|_{L^{\infty}}).
$$
Lemma~\ref{lem:rho_W1p} implies $\|\nabla D^{2}u[\rho_{\tau}^{n+1}]\|_{L^{2}}\le C\|\nabla\rho_{\tau}^{n+1}\|_{L^{2}}\le C(1+\mathcal{F}_{2}[\rho_{\tau}^{n+1}]).$ Together with Lemma~\ref{lem:control_phi} and Lemma~\ref{prop:JKO_FP} we obtain 
$$
|I|\le C \tau^2(1+\tau^{\beta}). 
$$

We focus on $C_{12}$, that we rewrite as
$$
C_{12} = \f{2}{\tau}\int_{\Omega}\rho_{\tau}^{n+1}\nabla\varphi_{\tau}^{n+1}\int_{0}^{1}D^{2}u[\rho_{\tau}^{n}-\rho_{\tau}^{n+1}](x-t\nabla\varphi_{\tau}^{n+1})\nabla\varphi_{\tau}^{n+1}\diff t \diff x.
$$

Therefore with a change of variable, Hölder inequality and the $L^{\infty}$ estimates provided by Theorem~\ref{lem:Linfty_JKO} and Lemma~\ref{lem:control_phi}: 
$$
C_{12}\le C \tau \mathcal{F}_{2}[\rho_{\tau}^{n+1}]^{1/2}\|D^{2}u[\rho_{\tau}^{n}-\rho_{\tau}^{n+1}]\|_{L^{2}} (1 + \tau^{\beta}).
$$

It remains to observe that 
$$
\|D^{2}u[\rho_{\tau}^{n}-\rho_{\tau}^{n+1}]\|_{L^{2}} \le \|\rho_{\tau}^{n} - \rho_{\tau}^{n+1}\|_{L^{2}} \le C \|\rho_{\tau}^{n}-\rho_{\tau}^{n+1}\|_{H^{-1}}\|\rho_{\tau}^{n}-\rho_{\tau}^{n+1}\|_{H^{1}}.
$$

Since both $\rho_{\tau}^{n}$ and $\rho_{\tau}^{n+1}$ are bounded in $H^1$ by Proposition~\ref{prop:JKO_FP} we can use Lemma~\ref{lem:sobone} and Lemma~\ref{lem:link_kantorovich} to deduce

$$
C_{12}\le C\tau^2 (1+\tau^\beta)\mathcal{F}_{2}[\rho]\le C\tau^2(1+\tau^{\beta}).
$$

We now focus on the remainder term. The first integral that we write as $R_1+ R_2$ where
\begin{align*}
& R_1 = 2\f{1}{\tau}\int_{\Omega}\rho_{\tau}^{n+1} \mathrm{Tr}[(D^2\varphi_{\tau}^{n+1})^2 \cdot(I-D^2\varphi_{\tau}^{n+1})^{-1}]\diff x, \\
& R_2 = 2\int_{\p\Omega}\rho_{\tau}^{n+1} \f{\nabla\varphi_{\tau}^{n+1}}{\tau}\cdot \vec{n} +  \rho_{\tau}^{n} \f{\nabla\psi_{\tau}^{n+1}}{\tau}\cdot\vec{n}\diff \mathcal{H}^{d-1}. 
\end{align*}

Using Lemma~\ref{lem:control_phi} we obtain
$$
R_1= 2\f{(1+\eps(\tau))}{\tau} \int_{\Omega} \rho_{\tau}^{n+1} |D^2\varphi_{\tau}^{n+1}|^2\diff x,
$$

We conclude
$$
R_1= 2(1+\eps(\tau))\tau \int_{\Omega} \rho_{\tau}^{n+1} |D^2(\log\rho_{\tau}^{n+1} - u[\rho_{\tau}^{n+1}])|^2\diff x. 
$$

The term $R_2$ can be adapted from the proof of~\cite[Theorem 6.3]{toshpulatov} and therefore we do not repeat the technical computations. We conclude
$$
R_2 = 2\tau\int_{\p\Omega}\rho_{\tau}^{n+1} \f{(\nabla (\log \rho_{\tau}^{n+1} - u[\rho_{\tau}^{n+1} ]))^{T} D^{2} h \nabla (\log \rho_{\tau}^{n+1} - u[\rho_{\tau}^{n+1} ])}{|\nabla h|}\diff \mathcal{H}^{d-1} + \tau\eps(\tau).
$$

Finally, summing over $n$ in all the terms yields the result. Indeed, since $\rho_{\tau}$ is the constant interpolation curve {\color{blue} defined in~\eqref{eq:def_JKO2}}, for all $f$: $\int_{0}^{T}f(\rho_{\tau})\diff t = \sum_{n}\tau f(\rho_{\tau}^{n}).$

\end{proof}

Before proving the convergence in $L^2(0,T;H^2(\Omega))$ of the curve we first prove strong convergence in $L^{p}(0,T; W^{1,p}(\Omega))$. 
\begin{proposition}\label{prop:first_convergence_JKO}
    {\color{blue}Let $\rho_{0}$ be an initial condition as in Theorem~\ref{thm:JKO}. Let $\rho_{\tau}$ be a constructed curve of the JKO scheme~\eqref{eq:def_JKO2}, bounded in $L^{\infty}((0,T)\times\Omega)$ uniformly in $\tau$}. Then $\rho_{\tau}$ is bounded uniformly in $\tau$ in $L^{2}(0,T; H^{2}(\Omega))$ and {\color{blue} up to a subsequence} $\rho_{\tau}\to \rho$ strongly  in $L^{p}(0,T; W^{1,p}(\Omega)) $.
\end{proposition}

\begin{proof}
{\color{blue}We begin with the $L^{2}(0,T;H^{2}(\Omega))$ estimate. By Proposition~\ref{prop:L^2H^2bound} and its proof, together with the fact that $\rho_\tau$ stays uniformly bounded away from zero (see Proposition~\ref{prop:JKO_maximum_principle}), the dissipation of the functional $\mathcal{F}_2[\rho]$ yields an
$L^{2}(0,T;L^{2}(\Omega))$ bound on
\[
D^{2}\bigl(\log\rho_\tau - u[\rho_\tau]\bigr).
\]
Indeed, the boundary term appearing in the dissipation has the good sign, while the remaining terms can be controlled using the $L^{p}(0,T;W^{1,p}(\Omega))$ bounds on $\rho_\tau$.

As a consequence, we obtain an $L^{2}(0,T;L^{2}(\Omega))$ bound on $D^{2}\log\rho_\tau$, since the term $D^{2}u[\rho_\tau]$ can be estimated by Calderón--Zygmund theory combined with Theorem~\ref{lem:Linfty_JKO}.} Finally, we note that
\[
D^{2}\log\rho_\tau
= \frac{1}{\rho_\tau} D^{2}\rho_\tau
- \frac{1}{\rho_\tau^{2}} \nabla\rho_\tau \otimes \nabla\rho_\tau,
\]
which allows us to conclude the desired $L^{2}(0,T;H^{2}(\Omega))$ estimate.
With Theorem~\ref{lem:Linfty_JKO} and Lemma~\ref{prop:JKO_FP}, Lemma~\ref{lem:rho_W1p} we conclude that $D^{2}\rho_{\tau}$ is bounded in $L^{2}(0,T; L^{2}(\Omega))$ uniformly in $\tau$. To prove the strong convergence in $L^{p}(0,T; W^{1,p}(\Omega))$ for all $1\le p<+\infty$, it is sufficient to prove that we have an $L^{2}(0,T; H^{1}(\Omega))$ strong convergence. The $L^{p}(0,T; W^{1,p}(\Omega))$ strong convergence follows then from the $L^{\infty}(0,T; W^{1,p}(\Omega))$ estimate from Proposition~\ref{prop:JKO_FP}. But the $L^{2}(0,T; H^{1}(\Omega))$ strong convergence follows from~\cite[Proposition 4.1]{toshpulatov}. The only difference is that here $V=-u[\rho^{n+1}_{\tau}]$ but for $q>d$ we can estimate $\text{Lip}(V)$ by
$$
\|\nabla u[\rho^{n+1}_{\tau}]\|_{L^{\infty}}\le C\|\nabla u[\rho^{n+1}_{\tau}]\|_{W^{1,q}} \le C\|\rho^{n+1}_{\tau}\|_{L^{q}}\le C\|\rho^{n+1}_{\tau}\|_{L^{\infty}}\le C. 
$$
where $C$ is independent of $n$ by Proposition~\ref{prop:JKO_maximum_principle}.
\end{proof}

We are now prepared to prove Theorem~\ref{thm:JKO}.

\begin{proof}[Proof of Theorem~\ref{thm:JKO}]
{\color{blue}We recall that $\rho$ is the weak solution of the Keller-Segel equation~\eqref{eq:KS}, the argument is classical as the penalization never activates. The estimates follow from Proposition~\ref{prop:JKO_FP} and Proposition~\ref{prop:first_convergence_JKO}. It remains to prove the $L^{2}(0,T; H^{2}(\Omega))$ strong convergence. Following and adapting the argument from~\cite[Theorem 6.5]{toshpulatov} to the Keller-Segel system, we have that up to a subsequence, $\rho_{\tau}$ converges strongly in $L^{p}(0,T; W^{1,p}(\Omega))$ for all $1\le p<+\infty$ and weakly in $L^{2}(0,T; H^{2}(\Omega))$ to some $\rho$; moreover $\log \rho_{\tau} -u[\rho_{\tau}]$ converges weakly to $\log \rho - u[\rho]$ weakly in $L^{2}(0,T; H^{2}(\Omega))$ and strongly in $L^{p}(0,T; W^{1,p}(\Omega))$. This uses the lower and upper bounds on $\rho_\tau$. Let us just mention than the convergence of $\rho_\tau$ implies the convergence $\log\rho_\tau - u[\rho_\tau]$ in general: $u[\rho_\tau]$ is a lower order, and $\log\rho_\tau$ is bounded and converges in the same space that $\rho_\tau$ as $\rho_\tau$ is uniformly bounded from below and above.}  By lower semi-continuity of the weak convergence and Proposition~\ref{prop:L^2H^2bound}:
\begin{align*}
 &\mathcal{F}_{2}[\rho_{0}] -  \mathcal{F}_{2}[\rho(T)]\ge \mathcal{F}_{2}[\rho_{0}] -  \liminf_{\tau\to 0}\mathcal{F}_{2}[\rho(T)] \ge 2\limsup_{\tau\to 0}\int_{0}^{T}\int_{\Omega} |D^{2}(\log \rho_{\tau}-u[\rho_{\tau}])|^2\rho_{\tau}\diff x \diff t \\
 &- 2\liminf_{\tau\to 0} \int_{0}^{T}\int_{\Omega}\rho_{\tau} \nabla(\log(\rho_{\tau} -u[\rho_{\tau}])\cdot D^{2}u[\rho_{\tau}]\nabla(\log(\rho_{\tau})-u[\rho_{\tau}])\diff x\diff t\\
 & - 2\liminf_{\tau\to 0} \int_{0}^{T}\int_{\Omega}\rho_{\tau}\nabla\left(\log(\rho_{\tau})-u[\rho_{\tau}]\right)\cdot \nabla u\left[\f{\rho_{\tau}(t)-\rho_{\tau}(t-\tau)}{\tau}\right]\diff x\diff t\\
 & +2\limsup_{\tau\to 0}\int_{0}^{T}\int_{\p\Omega}\rho \f{(\nabla \log \rho_{\tau}-u[\rho_{\tau}])\cdot D^{2} h \nabla (\log \rho_{\tau}-u[\rho_{\tau}])}{|\nabla h|}\diff \mathcal{H}^{d-1}\diff t + \eps(\tau). 
\end{align*}
With the convergence above we obtain 
\begin{align*}
&2\liminf_{\tau\to 0} \int_{0}^{T}\int_{\Omega}\rho_{\tau}  \nabla(\log(\rho_{\tau}) -u[\rho_{\tau}])\cdot D^{2}u[\rho_{\tau}]\nabla(\log(\rho_{\tau})-u[\rho_{\tau}]) \diff x\\
=&2 \int_{0}^{T}\int_{\Omega}\rho  \nabla(\log(\rho) -u[\rho]\cdot D^{2}u[\rho]\nabla(\log(\rho)-u[\rho])\diff x.
\end{align*}

Now we want to prove that 

\begin{align*}
&2\liminf_{\tau\to 0} \int_{0}^{T}\int_{\Omega}\rho_{\tau}\nabla\left(\log(\rho_{\tau})-u[\rho_{\tau}]\right)\cdot \nabla u\left[ \f{\rho_{\tau}(t)-\rho_{\tau}(t-\tau)}{\tau}\right]\diff x\diff t\\
=&\int_{0}^{T}\int_{\Omega}\rho\nabla(\log(\rho)-u[\rho_{\tau}])\cdot\nabla u[\p_t\rho]\diff t \diff x. 
\end{align*}

We already know that $\rho_{\tau}\to \rho$ in every $L^{p}(0,T; L^{p}(\Omega))$, $1\le p<+\infty$ by strong convergence of $\rho_{\tau}$ in $L^{2}(0,T; H^{1}(\Omega))$ and boundedness in $L^{\infty}((0,T)\times\Omega)$. Therefore it is sufficient to prove that $\nabla (\log\rho_{\tau} -u[\rho_{\tau}])\cdot\nabla u\left[ \f{\rho_{\tau}(t)-\rho_{\tau}(t-\tau)}{\tau}\right]$ converges weakly to $\nabla (\log \rho + u[\rho]) \cdot \nabla u[\p_{t}\rho] $ in $L^{2}(0,T; L^{2}(\Omega))$. 
But for $p>d$ with the embedding $W^{1,p}\hookrightarrow L^{\infty} $ and Lemma~\ref{prop:JKO_FP}
\begin{align*}
&\left\|\nabla (\log\rho_{\tau} -u[\rho_{\tau}])\cdot\nabla u\left[ \f{\rho_{\tau}(t)-\rho_{\tau}(t-\tau)}{\tau}\right]\right\|_{L^{2}(0,T; L^{2}(\Omega))} \\
&\le \|\nabla (\log\rho_{\tau} -u[\rho_{\tau}])\|_{L^{\infty}(0,T; W^{1,p}(\Omega))}
 \left\|\nabla u\left[\f{\rho_{\tau}(t)-\rho_{\tau}(t-\tau)}{\tau}\right]\right\|_{L^{2}(0,T; L^{2}(\Omega))}\\
 & \le C\left\|\f{\rho_{\tau}(t)-\rho_{\tau}(t-\tau)}{\tau}\right\|_{L^{2}(0,T; H^{-1}(\Omega))}.
\end{align*}
To estimate the quotients we recall that by the definition of the interpolation it is sufficient to estimate 
$$
\f{\rho^{n+1}_{\tau}-\rho^{n}_{\tau}}{\tau}
$$
in $H^{-1}(\Omega)$ where $\rho^{n}_{\tau}$ is the JKO sequence. But this follows from Lemma~\ref{lem:sobone} and Lemma~\ref{lem:sobone2} and Proposition~\ref{prop:JKO_FP}. Therefore the term $\nabla (\log\rho_{\tau} -u[\rho_{\tau}])\cdot \nabla u\left[\f{\rho_{\tau}(t)-\rho_{\tau}(t-\tau)}{\tau}\right]$ converges weakly in $L^{2}(0,T; L^{2}(\Omega))$. To identify its limit it is sufficient to do it against a smooth test function and using the strong convergence of $\nabla (\log \rho_\tau -u[\rho_{\tau}])$ in $L^{2}(0,T; L^{2}(\Omega))$. Therefore
\begin{align*}
&2\liminf_{\tau\to 0} \int_{0}^{T}\int_{\Omega}\rho_{\tau}\nabla(\log(\rho_{\tau})-u[\rho_{\tau}])\nabla u\left[\f{\rho_{\tau}(t)-\rho_{\tau}(t-\tau)}{\tau}\right]\\
=&\int_{0}^{T}\int_{\Omega}\rho(x)\nabla(\log(\rho)-u[\rho])\nabla u[\p_{t}\rho]. 
\end{align*}

Finally by weak convergence of $\log\rho_{\tau} -u[\rho_{\tau}]$ in $L^{2}(0,T; H^{2}(\Omega))$, we have weak convergence of $\nabla (\log \rho_{\tau} -u[\rho_{\tau}])$ in $L^{2}(0,T; L^{2}(\p\Omega))$. Therefore

\begin{align*}
&\limsup_{\tau\to 0}2\int_{\p\Omega}\rho_{\tau} \f{(\nabla (\log \rho_{\tau}-u[\rho_{\tau}]))^{T} D^{2} h \nabla (\log \rho_{\tau}-u[\rho_{\tau}])}{|\nabla h|}\diff \mathcal{H}^{d-1} \\
&\ge 2\int_{\p\Omega}\rho \f{(\nabla (\log \rho-u[\rho]))^{T} D^{2} h \nabla (\log \rho-u[\rho])}{|\nabla h|}\diff \mathcal{H}^{d-1}
\end{align*}

In the end together with Proposition~\ref{prop:entropy_equality} with $p=2$ we deduce
$$
\int_{0}^{T}\int_{\Omega} |D^{2}(\log \rho-u[\rho])|^2\rho\diff x \diff t \ge \limsup_{\tau\to 0}\int_{0}^{T}\int_{\Omega} |D^{2}(\log \rho_{\tau}-u[\rho_{\tau}])|^2\rho_{\tau}\diff x \diff t.  
$$
Combining it with the weak lower semi continuity of the norm and the weak convergence in $L^{2}(0,T; H^{2}(\Omega))$ of $\log \rho_{\tau} -u[\rho_{\tau}]$
$$
 \lim_{\tau\to 0}\int_{0}^{T}\int_{\Omega} |D^{2}(\log \rho_{\tau} -u[\rho_{\tau}])|^2\rho_{\tau}\diff x \diff t =\int_{0}^{T}\int_{\Omega} |D^{2}(\log \rho-u[\rho])|^2\rho\diff x \diff t.
$$
Since $\rho_{\tau}$ is bounded from below we deduce that $D^{2}(\log\rho_{\tau}-u[\rho_{\tau}])$ converges weakly in $L^{2}(0,T; H^{2}(\Omega))$ to $D^{2}(\log \rho -u[\rho])$. $\sqrt{\rho_{\tau}}$ converges strongly in all $L^{p}(0,T; L^{p}(\Omega))$ by convergence a.e. and boundedness in $L^{\infty}((0,T)\times\Omega)$. Therefore $\sqrt{\rho_{\tau}}D^{2}(\log\rho_{\tau} -u[\rho_{\tau}])$ converges weakly in $L^{2}(0,T; L^{2}(\Omega))$ to $\sqrt{\rho_{\tau}}D^{2}(\log\rho -u[\rho])$.  Thus we can deduce to the strong convergence in $L^{2}(0,T; H^{2}(\Omega))$ by weak convergence and convergence of the norms. Multiplying by $\rho_{\tau}^{-1/2}$ and substracting $D^{2}u[\rho_{\tau}]$ (which converges strongly since $\rho_{\tau}$ converges strongly in $L^{2}(0,T ;H^{1}(\Omega))$ we deduce the strong convergence of $\log \rho_{\tau}$ and therefore of $\rho_{\tau}$ in $L^{2}(0,T; H^{2}(\Omega)).$ 
\end{proof}

\section*{Acknowledgments}
This project was supported by the European Union via the ERC
AdG 101054420 EYAWKAJKOS project.

\bibliographystyle{siam}
\bibliography{biblio}

\end{document}